\documentclass{amsart}

\usepackage{amscd,amsmath,amssymb,amsfonts,bbm, calligra, xspace}
\usepackage[T1]{fontenc}
\usepackage{lmodern}
\usepackage{mathtools}
\usepackage{enumerate}
\usepackage{multicol}
\usepackage[all]{xy}
\usepackage{mathrsfs}
\usepackage{footmisc}

\usepackage{xcolor}
\definecolor{darkblue}{rgb}{0,0,0.8}
\definecolor{darkgreen}{rgb}{0,0.4,0}
\usepackage[colorlinks=true,linkcolor=darkblue, citecolor=darkgreen,  urlcolor=darkgreen, unicode,pdfborder={0 0 0},final]{hyperref}

\newtheorem{thm}{Theorem}[section]
\newtheorem{prop}[thm]{Proposition}

\newtheorem{lem}[thm]{Lemma}
\newtheorem{cor}[thm]{Corollary}

\theoremstyle{definition}

\theoremstyle{remark}
\newtheorem{rem}[thm]{Remark}
\newtheorem{rems}[thm]{Remarks}

\newtheorem{Step}{Step}

\numberwithin{equation}{section}

\newcommand{\car}{\mathrm{char}} 

\newcommand{\Id}{\mathrm{Id}}

\newcommand{\Ker}{\mathrm{Ker}}

\newcommand{\Hom}{\mathrm{Hom}}
\newcommand{\cl}{\mathrm{cl}}

\newcommand{\qc}{\mathrm{qc}}

\newcommand{\qfh}{\mathrm{qfh}}

\newcommand{\N}{\mathrm{N}}
\newcommand{\cd}{\mathrm{cd}}

\newcommand{\sing}{\mathrm{sing}}

\newcommand{\colim}{\mathrm{colim}}
\newcommand{\et}{\mathrm{\acute{e}t}}
\newcommand{\red}{\mathrm{red}}

\newcommand{\pr}{\mathrm{pr}}

\newcommand{\an}{\mathrm{an}}

\newcommand{\Spec}{\mathrm{Spec}}

\newcommand{\Frac}{\mathrm{Frac}}

\newcommand{\Gal}{\mathrm{Gal}}

\newcommand{\RR}{\mathrm{R}}

\newcommand{\isoto}{\myxrightarrow{\,\sim\,}}
\makeatletter
\def\myrightarrow{{\setbox\z@\hbox{$\rightarrow$}\dimen0\ht\z@\multiply\dimen0 6\divide\dimen0 10\ht\z@\dimen0\box\z@}}
\def\myrightarrowfill@{\arrowfill@\relbar\relbar\myrightarrow}
\newcommand{\myxrightarrow}[2][]{\ext@arrow 0359\myrightarrowfill@{#1}{#2}}
\makeatother

\def\loccit{\emph{loc}.\kern3pt \emph{cit}.{}\xspace}
\def\eg{e.g.\kern.3em}
\def\ie{i.e.,\ }
\def\resp {\text{resp.}\kern.3em}

\def\A{\mathbb A}

\def\Z{\mathbb Z}
\def\C{\mathbb C}

\def\L{\mathbb L}
\def\M{\mathbb M}

\def\P{\mathbb P}
\def\R{\mathbb R}
\def\N{\mathbb N}

\def\cN{\mathcal{N}}
\def\cM{\mathcal{M}}

\def\cO{\mathcal{O}}

\def\cF{\mathcal{F}}

\def\cH{\mathcal{H}}

\def\cM{\mathcal{M}}
\def\cS{\mathcal{S}}
\def\cU{\mathcal{U}}

\def\cI{\mathcal{I}}
\def\ci{\mathcal{C}^{\infty}}

\def\whK{\widehat{K}}
\def\whR{\widehat{R}}

\def\hh{\hat{h}}

\def\hq{\hat{q}}

\def\kp{\mathfrak{p}}
\def\kq{\mathfrak{q}}
\def\km{\mathfrak{m}}

\def\tf{\tilde{f}}

\def\ti{\tilde{i}}
\def\tj{\tilde{j}}

\def\ts{\tilde{s}}

\def\talpha{\tilde{\alpha}}
\def\tbeta{\tilde{\beta}}

\def\oY{\overline{Y}}
\def\oX{\overline{X}}

\def\oW{\overline{W}}

\def\wS{\widetilde{S}}

\def\wU{\widetilde{U}}

\def\wX{\widetilde{X}}
\def\wY{\widetilde{Y}}
\def\wZ{\widetilde{Z}}

\begin{document}

\title[]{\'Etale cohomology of algebraic varieties over Stein compacta}

\author{Olivier Benoist}
\address{D\'epartement de math\'ematiques et applications, \'Ecole normale sup\'erieure, CNRS,
45 rue d'Ulm, 75230 Paris Cedex 05, France}
\email{olivier.benoist@ens.fr}

\renewcommand{\abstractname}{Abstract}
\begin{abstract}
We prove a comparison theorem between the \'etale cohomology of algebraic varieties over Stein compacta and the singular cohomology of their analytifications. We deduce that the field of meromorphic functions in a neighborhood of a connected Stein compact subset of a normal complex space of dimension~$n$ has cohomological dimension~$n$. As an application of $\Gal(\C/\R)$\nobreakdash-equivariant variants of these results,  we obtain a quantitative version of Hilbert's 17th problem on compact subsets of real\nobreakdash-analytic spaces.
\end{abstract}
\maketitle

\section*{Introduction}

\subsection{Hilbert's 17th problem}

Hilbert discovered in \cite{Hilbert} that real polynomials $f\in\R[x_1,\dots,x_n]$ that are positive semidefinite (\ie that only take nonnegative values on $\R^n$) may not always be written as sums of squares of polynomials. He however conjectured, in his celebrated $17$th problem, that this issue may be resolved by allowing denominators: the polynomial~$f$ should always be a sum of squares of rational functions.
This problem was solved positively by E. Artin \cite[Satz~4]{Artin} in~1927.  
Forty years later, Pfister \cite[Theorem 1]{Pfister} obtained a quantitative improvement of Artin's theorem bounding the number of squares required. It states that a positive semidefinite $f\in\R[x_1,\dots,x_n]$ is a sum of $2^n$ squares in $\R(x_1,\dots,x_n)$.

\subsection{The real-analytic variant}

As an application of the main results of this article, we prove a quantitative theorem \`a la Pfister in real-analytic geometry. 

\begin{thm}[Theorem \ref{sosreal}]
\label{th1}
Let $M$ be a normal real-analytic variety of pure dimension $n$. Let $K\subset M$ be compact.  
Let $f:M\to \R$ be a nonnegative real-analytic function. Then there exists an open neighborhood $U$ of $K$ in $M$ such that~$f|_U$ is a sum of $2^n$ squares of real-analytic meromorphic functions on $U$.
\end{thm}

If $M$ is compact and $K=M$, Theorem \ref{th1} states that a nonnegative real-analytic function on $M$ is a sum of $2^n$ squares of real-analytic meromorphic functions on~$M$.

Much less is known on Hilbert's 17th problem in the real-analytic setting, first considered in \cite{BR}
 (see \cite{ABFbook} for an up-to-date survey), than in the algebraic case.
A nonnegative real-analytic function on a normal and pure-dimensional real-analytic variety $M$
is a sum of squares of real-analytic meromorphic functions  in dimension $\leq 2$ (see \cite[Corollary 2]{Jaworski1}
and \cite[Theorem 1]{ADR}),  
under a compactness hypothesis (see \cite[Theorem~1]{Ruiz} and \cite[Theorem~1]{Jaworski2})
or, when~$M$ is a manifold, under a discreteness hypothesis on the zeros of~$f$ \cite[Theorem~1]{BKS}.
However, whether a nonnegative real-analytic function on~$\R^n$ is a sum of squares of real-analytic meromorphic functions is still an open problem for~$n\geq 3$.

In addition, quantitative results are known in dimension $\leq 2$.  A nonnegative real-analytic function on a normal real-analytic surface $M$ is a sum of $3$ squares of real-analytic functions if $M$ is a manifold \cite[Corollary 2]{Jaworski1}, and a sum of $5$ squares of real-analytic meromorphic functions in general (see \cite[Theorem~1.2]{ABFR1} or the more general \cite[Theorem 1.3]{Fernando}).
No bounds on the required number of squares were known if $n\geq 3$ (except in the local case, for which see \cite[Theorem~0.2]{Henselian}). Theorem \ref{th1} rectifies this situation, under a compactness hypothesis.

Although the qualitative content of Theorem~\ref{th1} is not new (see \cite[Theorem~1]{Jaworski2}), the quantitative bounds it provides (more precisely, that these bounds only depend on the dimension of $M$) do imply new cases of (the qualitative version of) the real-analytic Hilbert's 17th problem. Indeed, work of Acquistapace, Broglia, Fernando and Ruiz \cite[Proposition~1.8]{ABFR2} readily implies the following.

\begin{cor}
\label{cor1}
Let $f$ be a nonnegative real-analytic function on a real-analytic manifold. If the zero set of $f$ is a disjoint union of compact sets, then $f$ is a finite sum of squares of real-analytic meromorphic~functions.
\end{cor}

The best results to date in this vein handled the case where the zero locus of~$f$ is the union of a compact set and of a discrete set (see \cite[Theorem~2]{Jaworski3} or \cite[Corollary~1.10]{ABFR2}), or concerned infinite sums of squares \cite[Corollary 1.9]{ABFR2}.
We also refer to \cite[Theorem 1.1]{ABFmult} for quantitative results in this direction.

\subsection{From real-analytic to complex-analytic geometry}

In real algebraic geometry, it is very important to consider not only the sets of real points of real algebraic varieties,  but also their sets of complex points, endowed with the action of $G:=\Gal(\C/\R)$ by complex conjugation. For the same reason, it is crucial for our proof of Theorem~\ref{th1} to work in the setting of $G$-equivariant complex-analytic geometry (see \S\ref{Geqpar} for our conventions) instead of real-analytic geometry.

This point of view also leads to
a more general theorem on sums of squares, from which Theorem~\ref{th1} is easily derived using the successive works of Cartan \cite{Cartananal}, Grauert \cite{GrauertLevi} and Tognoli \cite{Tognoli} on complexifications of real-analytic spaces. Recall that a compact subset of a complex space is said to be Stein if it admits a basis of Stein open neighborhoods.

\begin{thm}[Theorem \ref{soscx}]
\label{th2}
Let $S$ be a reduced $G$-equivariant Stein space of dimension $n$.  Let $K\subset S$ be a $G$\nobreakdash-invariant Stein compact subset.  Any $G$-invariant holomorphic function on $S$ which is nonnegative on $S^G$ is a sum of~$2^n$ squares of $G$-invariant meromorphic functions in a neighborhood of $K$.
\end{thm}

We believe that this $G$-equivariant variation on the analytic Hilbert's 17th problem is novel, and that the new statements it comprises (for instance when $S=\C^n$ and $K\subset S$ is the closed unit ball) are of interest. We insist that Theorem \ref{th2} is already interesting and nontrivial when $S^G=\varnothing$, in which case the nonnegativity hypothesis is automatically satisfied.

\subsection{Cohomological dimension of fields of meromorphic functions}

It has been known since Voevodsky's proof of the Milnor conjectures \cite{Voevodsky} that quadratic forms over a field are largely governed by the Galois cohomology of the field.  As a consequence
of these results, the validity of Theorem \ref{th2} is controlled by the vanishing of a single Galois cohomology class of degree $n+1$ (see \cite[Proposition~2.1]{Henselian}). 
Using this point of view,  Theorem \ref{th2} is a consequence of an upper bound for the cohomological dimension of fields of $G$-invariant meromorphic functions on $G$-equivariant Stein compacta (Corollary \ref{cdfieldGcor}).
In this introduction,  let us only state a non-$G$-equivariant version of this theorem.

\begin{thm}[Corollary \ref{cdfieldcor} and Remark \ref{remcd} (ii)]
\label{th3}
Let $K$ be a connected Stein compact subset of a normal Stein space~$S$ of dimension $n$.  The field $\cM(K)$ of germs of meromorphic functions in a neighborhood of $K$ has cohomological dimension $n$.
\end{thm}

In dimension $1$, Theorem \ref{th3} is attributed to M. Artin by Guralnick (\cite[Proposition 3.7]{Guralnick}, see also \cite[Proposition A.6]{tight}), with no compactness hypothesis. It is new in dimension~$\geq 2$. 
Theorem \ref{th3} is an algebraic statement concerning a field of analytic origin, and its proof uses a mixture of analytic and algebraic tools.  As such, it is deeper than its algebraic counterpart \cite[II, Proposition~11]{CG} bounding the cohomological dimension of function fields of algebraic varieties. 

We also obtain bounds on the \'etale cohomological dimension of affine varieties over (possibly $G$-equivariant) Stein compacta (see Theorems \ref{cdaffineG} and \ref{cdaffine}).
We refer to the work of Bhatt and Mathew \cite[Theorem~7.3 and Remark 7.4]{BM} for related bounds in non-archimedean analytic geometry.

\subsection{\'Etale and singular cohomology}

By a complex-analytic incarnation of the weak Lefschetz theorem due to Andreotti and Frankel \cite[\S 2]{AF} in the nonsingular case and to Hamm \cite[Satz 1]{Hamm} in general,  a Stein space of dimension~$n$ has the homotopy type of a CW complex of dimension~$\leq n$.
Its singular cohomology therefore vanishes in degree $>n$.
Our strategy of proof of Theorem \ref{th3} is to transfer this topological information on the singular cohomology of a Stein compactum to algebraic information on the \'etale cohomology of its ring of holomorphic functions, and eventually on the Galois cohomology of its field of meromorphic functions.

 In algebraic geometry, such a transfer tool exists in the shape of M.~Artin's comparison theorem \cite[XVI, Th\'eor\`eme 4.1]{SGA43} between the \'etale cohomology of a complex algebraic variety $X$ and the singular cohomology of its analytification~$X^{\an}$.
In order to implement the strategy described above, we prove the following analogue of Artin's comparison theorem in analytic geometry.

\begin{thm}[Theorem \ref{Artincompcx}]
\label{th4}
Let $S$ be a Stein space. Let $X$ be an $\cO(S)$\nobreakdash-scheme of finite type and let $\L$ be a torsion \'etale abelian sheaf on $X$. Assume that $X$ is proper over $\cO(S)$ or that $\L$ is constructible. 
If one lets $U$ run over all Stein open neighborhoods of a Stein compact subset $K$ of $S$, the change of topology morphisms
\begin{equation*}
\underset{K\subset U}{\colim}\,\,H^k_{\et}(X_{\cO(U)},\L_{\cO(U)})\to\underset{K\subset U}{\colim}\,\,H^k((X_{\cO(U)})^{\an},\L^{\an})
\end{equation*}
are isomorphisms for $k\geq 0$.
\end{thm}

Theorem \ref{th4} is our main result. The bulk of this article is dedicated to its proof.
For applications to sums of squares, we really need a $G$-equivariant extension of Theorem \ref{th4} (see Theorem \ref{Artincompreal}), which we deduce from Theorem \ref{th4} by means of the Hochschild--Serre spectral sequence.

\subsection{Proof of the comparison theorem}

Two proofs of Artin's comparison theorem appear in \cite{SGA43}. The first \cite[XI, \S 4]{SGA43} (which only works for smooth varieties and locally constant coefficients) compares the \'etale and the classical topology by means of a Leray spectral sequence. It exploits the fact that points in smooth complex varieties admit good neighborhoods, which are iterated fibrations in smooth affine curves. 
The second \cite[XVI, \S 4]{SGA43} uses extensive d\'evissage and fibration arguments to reduce to the case of smooth projective curves. 

None of these proofs adapt to the setting of Theorem \ref{th4} as such fibration arguments cannot be successfully implemented in Stein geometry.
However, the d\'evissage argument of the second proof can indeed be used in the complex-analytic setting to yield the following easier \textit{relative} comparison theorem over Stein spaces.

\begin{thm}[Theorem \ref{relArtin}]
\label{th5}
Let $S$ be a Stein space.
Let $f:X\to Y$ be a morphism of $\cO(S)$-schemes of finite type and let $\L$ be a torsion \'etale abelian sheaf on $X$. Assume that $f$ is proper or that $\L$ is constructible.  Then the base change morphisms $\RR^k f_*(\L)^{\an}\to\RR^k f^{\an}_*(\L^{\an})$ are isomorphisms for $k\geq 0$.
\end{thm}

In contrast, our proof of Theorem \ref{th4} initially takes the point of view of the first proof of Artin's theorem. Reductions based on Theorem \ref{th5} allow us to assume that $X=\Spec(\cO(S))$ and $\L=\Z/m$. We then consider the Leray spectral sequence comparing the \'etale and the classical topology. In order to show that this spectral sequence degenerates, one has to prove that singular cohomology classes with $\Z/m$ coefficients on analytifications of \'etale $\cO(S)$-schemes are \'etale-locally trivial.

We proceed in two steps. First, we show that such cohomology classes become unramified after pull-back to a finite ramified covering (Proposition~\ref{thkillram}). This uses Theorem \ref{th5} in an essential way, as well as a vanishing theorem of Bhatt \cite[Theorem 1.1]{Bhatt}. Second, we show that such unramified cohomology classes may be killed by a further finite ramified covering (Proposition~\ref{killunram}). This second step is analytic in nature. It relies on the Oka--Weil approximation theorem and on Grauert's bump method as developed by Henkin and Leiterer~\cite{HL} and Forstneri\v{c}~\cite{Forstneric}.

This procedure unfortunately only shows that the relevant singular cohomology classes are killed on possibly ramified coverings, not on \'etale coverings. To overcome this difficulty, our proof of Theorem \ref{th4} makes use of Voevodsky's qfh topology \cite[\S 3.1]{Voehomo}, which is finer than the \'etale topology and allows for such coverings.

\vspace{.5em}

Our relative comparison theorem (Theorem \ref{th5}) is an analogue for \'etale sheaves of the relative GAGA theorem of Hakim \cite[VIII, Th\'eor\`eme~3.2]{Hakim} (see also \cite[Theorem 4.2]{Bingener} or \cite[Appendix C]{AT}) for coherent sheaves. 
In the coherent context, there is no need for an \textit{absolute} comparison theorem such as Theorem~\ref{th4}, as higher cohomology groups of algebraic coherent sheaves on affine schemes, or of analytic coherent sheaves on Stein spaces, vanish. This is a marked difference with the \'etale setting, in which the absolute comparison theorem lies deeper.

In this article, we make extensive use of relative algebraic geometry over complex spaces as initiated by Hakim \cite{Hakim},
and developed by Bingener \cite{Bingener} when the base is Stein. We refer to \S\ref{parBingener} for our conventions which follow \cite{Bingener}.
In particular, we use the above-mentioned relative GAGA theorem, through an application of \cite[(7.2)]{Bingener}, in Step \ref{step4bis} of the proof of Theorem~\ref{th4}.

\subsection{Structure of the article}

 Section \ref{parStein} gathers general results concerning Stein spaces. Most of this material is included to fix our conventions, for the convenience of the reader, or for lack of appropriate references. Specialists of Stein geometry might want to skip it. Section \ref{parunr} contains the main analytic input of our work: a procedure to kill a singular cohomology class with torsion coefficients on a Stein compactum, after pull-back to a finite ramified covering. 
We deduce Theorem~\ref{th5} from Artin's comparison theorem and its proof in Section \ref{parcomprel},  and we use it to kill the ramification of singular cohomology classes with torsion coefficients on finite ramified coverings in Section \ref{parram}. The above results are combined in Section~\ref{parcomp} to prove Theorem \ref{th4}, and to derive cohomological dimension bounds including Theorem \ref{th3}. Finally, applications to sums of squares problems are given in Section~\ref{parsos}.

\subsection{Acknowledgements}

I thank Jean-Beno\^it Bost for a helpful conversation and for having made \cite{Bingener} available to me, and Jos\'e F. Fernando and Marco Maculan for useful comments. I am also grateful to an anonymous referee for suggestions which significantly improved the exposition of the paper.

\section{Geometry of Stein spaces}
\label{parStein}

Among the many general facts on Stein spaces and their compact subsets that are collected in this section, let us put forward a descent result for $\cO(S)$-convexity along finite surjective holomorphic maps (Proposition \ref{propfinitehull}), and the correspondence between analytic coverings of connected normal Stein spaces and finite extensions of their meromorphic function fields (Proposition \ref{eqcat}).

\subsection{Complex spaces and coherent sheaves}

  A \textit{complex space} is a $\C$-ringed space that is locally isomorphic to a model complex space defined by a finitely generated ideal sheaf in a domain of $\C^N$ (see \cite[1, \S 1.5]{GRCoherent}).
We assume that they are second-countable, but not necessarily Hausdorff,  finite-dimensional or reduced.

  As in \cite{GRStein}, a complex space $S$ is said to be \textit{Stein} if $H^k(S,\cF)=0$ for all $k\geq 1$ and all coherent sheaves $\cF$ on $S$.  
We refer to this property as Cartan's Theorem B and to its consequence that coherent sheaves on Stein spaces are globally generated as Cartan's Theorem A.  Stein spaces are Hausdorff (holomorphic functions separate distinct $s, s'\in S$ as $H^1(S,\cI_{\{s,s'\}})=0$).
A complex space $S$ is Stein if and only if its reduction $S^{\red}$ is \cite[V, \S 4.3, Theorem 5]{GRStein}. The following lemma sometimes allows us to reduce problem about Stein spaces to the case of Stein manifolds.

\begin{lem}
\label{lemretract}
Let $S$ be a Stein space of dimension $n$. Then there exist a Stein manifold $S'$ of dimension $2n+1$ and a proper injective holomorphic map $i:S\to S'$ such that $i(S)$ is a strong deformation retract of $S'$.
\end{lem}

\begin{proof}
Let $i:S^{\red}\to \C^{2n+1}$ be  a proper injective holomorphic map (see \cite[Theorem 5]{Narasimhan}). As $\cO(S)\to\cO(S^{\red})$ is onto (by Cartan's Theorem B), one may extend it to a proper injective holomorphic map $i:S\to \C^{2n+1}$.
Its image ${i(S)\subset \C^{2n+1}}$ is a Stein analytic subset (by \cite[3, \S 1.3, Proposition]{GRCoherent} and \cite[V, \S 1, Theorem~1~b)]{GRStein}). The existence of $S'$ now follows from \cite[Theorem~3.1]{Mihalache}.
\end{proof}

If $S$ is a complex space and $\cF$ is a coherent sheaf on $S$, we endow $H^0(S,\cF)$ with the canonical Fr\'echet topology defined in \cite[V, \S 6]{GRStein} and use freely its properties listed in \cite[V, \S 6.4]{GRStein}. 
When $S$ is reduced and $\cF=\cO_S$, it coincides with the topology of uniform convergence on compact subsets \cite[V, \S 6.6, Theorem 8]{GRStein}. 

\subsection{Runge domains}

If $K\subset S$ is a compact subset of a complex space, its $\cO(S)$\nobreakdash-\textit{convex hull} is
$\whK_{\cO(S)}:=\{s\in S\mid |f(s)|\leq\sup_{t\in K}|f(t)|\textrm{ for all }f\in\cO(S)\}.$
By \cite[V, \S 4.2, Theorem 3]{GRStein}, a complex space $S$ is Stein if and only if
\begin{enumerate}[(i)]
\item the global holomorphic functions $\cO(S)$ separates the points of $S$, and
\item the $\cO(S)$-convex hull $\whK_{\cO(S)}\subset S$ of any compact subset $K\subset S$ is compact.
\end{enumerate}

An open subset $\Omega$ in a Stein space $S$ is said to be \textit{Runge} if it satisfies any of the equivalent properties of the following proposition. As we could not find a precise discussion in the literature when $S$ is possibly nonreduced, we include a brief proof.
Such a generality will be useful in the proof of Proposition \ref{propfinitehull}.

\begin{prop}
\label{OkaWeilnonreduit}
Let $\Omega\subset S$ be an open subset in a Stein space. The following assertions are equivalent, and hold for $(S,\Omega)$ if and only if they hold for $(S^{\red},\Omega^{\red})$.
\begin{enumerate}[(i)]
\item For all compact subsets $K\subset \Omega$, one has $\whK_{\cO(\Omega)}=\whK_{\cO(S)}$.
\item The image of the restriction map $\cO(S)\to\cO(\Omega)$ is dense, and $\Omega$ is Stein.
\item For any coherent sheaf $\cF$ on $S$, the restriction map $H^0(S,\cF)\to H^0(\Omega,\cF)$ has dense image, and $\Omega$ is Stein.
\end{enumerate}
\end{prop}

\begin{proof}
Write $(*)^{\red}$ for assertion $(*)$ for the pair $(S^{\red},\Omega^{\red})$. Then (i)$^{\red}$, (ii)$^{\red}$ and (iii)$^{\red}$ are equivalent by \cite[VII, A, Corollary 9 and VIII, A, Theorem~11]{GunningRossi}. The implication (iii)$\Rightarrow$(ii) is obvious. If (ii) holds, the continuity and surjectivity of the restriction map $\cO(\Omega)\to\cO(\Omega^{\red})$ show that (ii)$^{\red}$ also holds. That (i)$^{\red}$ implies (i) follows from the surjectivity of $\cO(S)\to\cO(S^{\red})$. Conversely, if (i) holds, then~$\Omega$ is Stein ($\cO(\Omega)$ separates the points of $\Omega$ because $\cO(S)$ separates the points of $S$) and (i)$^{\red}$ follows from the surjectivity of $\cO(\Omega)\to\cO(\Omega^{\red})$. 

To prove (iii)$^{\red}\Rightarrow$(iii), we argue as in \cite[Lemma 1.10]{NorguetSiu}.
The characterization \cite[V, \S 6.2, Theorem 4]{GRStein} of the topology on $H^0(\Omega,\cF)$ shows that it is the initial topology with respect to the restriction maps $H^0(\Omega,\cF)\to H^0(\Omega',\cF)$ with $\Omega'\subset \Omega$ relatively compact. By \cite[(1.1)]{NarasimhanLeviII}, we may restrict to those $\Omega'$ such that~$(\Omega')^{\red}$ is Runge in $\Omega^{\red}$. We may thus replace $\Omega$ with such an $\Omega'$ and assume that~$\Omega$ is relatively compact in $S$. Let $\cN\subset \cO_S$ be the nilradical of $\cO_S$. As $\Omega$ is relatively compact in $S$, there exists $k\geq 0$ such that $\cN^k=0$ on $\Omega$.
Then $\cF$ is a successive extension of $\cF/\cN\cF,\dots, \cN^{k-1}\cF/\cN^k\cF$ and $\cN^k\cF$. As the set of coherent sheaves satisfying (iii) is stable by extensions (use \cite[Lemma~1.9]{NorguetSiu}),  we may suppose that $\cN\cF=0$ on $\Omega$. As $H^0(S,\cF)\to H^0(S,\cF/\cN\cF)$ is onto by Cartan's Theorem~B, we may replace $\cF$ with $\cF/\cN\cF$. Now $\cF$ may be identified with a coherent sheaf on~$S^{\red}$ and the required density statement follows from~(iii)$^{\red}$.
\end{proof}

\subsection{\texorpdfstring{$\cO(S)$}{O(S)}-convex compact subsets}
\label{Oconvex}

A compact subset $K$ of a complex space~$S$ is said to be $\cO(S)$\nobreakdash-\textit{convex} if $\whK_{\cO(S)}=K$. 
In the following lemma, $\ci$ strongly plurisubharmonic (psh) functions are meant in the sense of \cite[\S 2]{NarasimhanLeviI}.

\begin{lem}
\label{critpsh}
Let $\rho:S\to \R$ be a $\ci$
strongly psh function on a reduced
Stein~space.
\begin{enumerate}[(i)]
\item For $c\in\R$, the open subset $\{s\in S\mid\rho(s)< c\}$ of $S$ is Runge.
\item For $c\in\R$, if $S_{\leq c}:=\{s\in S\mid\rho(s)\leq c\}$ is compact, then it is $\cO(S)$-convex.
\end{enumerate}
\end{lem}

\begin{proof}
\label{lemRunge}
Assertion (i) is \cite[Theorem 3]{NarasimhanLeviI}. 
The compactness of $S_{\leq c}$ and (i) together imply that $S_{\leq c}$ has a basis of Runge neighborhoods in $S$, hence is $\cO(S)$\nobreakdash-convex.
\end{proof}

We also include the next lemma, proven in \cite[(1.1)]{NarasimhanLeviII} for later reference.

\begin{lem}
Let $S$ be a Stein space and let $K\subset S$ be an $\cO(S)$-convex compact subset. Then $K$ admits a basis of Runge open neighborhoods in $S$. 
\end{lem}

We now analyze how $\cO(S)$-convexity behaves with respect to finite morphisms. Our goal is Proposition \ref{propfinitehull} (ii) which will be used in the proof of Proposition \ref{killunram}.

\begin{lem}
\label{lemnormal}
Let $p:T\to S$ be a finite surjective holomorphic map between connected normal Stein spaces. Let $K\subset S$ be a compact subset. Then 
$$\widehat{p^{-1}(K)}_{\cO(T)}=p^{-1}(\whK_{\cO(S)}).$$
\end{lem}

\begin{proof}
The inclusion $\widehat{p^{-1}(K)}_{\cO(T)}\subset p^{-1}(\whK_{\cO(S)})$ is obvious. To prove the converse inclusion, 
we use the \textit{norm morphism} $\textrm{N}_p:p_*\cO_{T}\to\cO_S$ of $p$ (it is defined on the locus over which $p$ is an unramified finite covering of manifolds by summing 
over the fibers and it extends to all of $S$ by Riemann's extension theorem \cite[7,~\S~4.2]{GRCoherent} because $S$ is normal).

 Choose $t\in T\setminus \widehat{p^{-1}(K)}_{\cO(T)}$ and set $s:=p(t)$.  Let $\delta$ be the maximal cardinality of the fibers of $p$ over some compact neighborhood of $K$. By the version \cite[Theorem~2.8.4]{Forstneric} of the Oka--Weil approximation theorem, there exists $g\in \cO(T)$ with $|g|\leq 1$ on $\widehat{p^{-1}(K)}_{\cO(T)}$, with $g=0$ on $p^{-1}(\{s\})\setminus\{t\}$ and with $g(t)=\delta+1$. Then $f:=\textrm{N}_p(g)\in \cO(S)$ satisfies $|f|\leq \delta$ on~$K$, and $|f(s)|> \delta$ (to see it, use that~$p$ is open by \cite[3, \S 3.2, Criterion of Openness]{GRCoherent}). 
Hence $s\notin\whK_{\cO(S)}$.
\end{proof}

\begin{prop}
\label{propfinitehull}
Let $p:T\to S$ be a finite holomorphic map between Stein spaces, and let $K\subset S$ be a compact subset.
\begin{enumerate}[(i)]
\item If $K$ is $\cO(S)$-convex, then $p^{-1}(K)$ is $\cO(T)$\nobreakdash-convex.
\item If $p^{-1}(K)$ is $\cO(T)$\nobreakdash-convex and $p$ is surjective, then $K$ is $\cO(S)$-convex.
\end{enumerate}
\end{prop}

\begin{proof}
Assertion (i) is immediate from the definitions and we now prove (ii). Using~(i), one may replace $T$ with its normalization and hence assume that $T$ is normal. We may also replace $S$ with its reduction and assume that it is reduced. 

Choose $s\in S\setminus K$. To construct $f\in\cO(S)$ such that $|f|<1$ on~$K$ and $|f(s)|=1$, we may first construct it in restriction to the (finite) union of irreducible components of $S$ intersecting $K\cup\{s\}$ and then extend it to $S$ using Cartan's Theorem B. We may thus assume that $S$ has finite dimension $n$ and argue by induction on $n$.

Let~$\wS$ be the normalization of $S$. Using the natural factorization $T\to\wS\to S$ of~$p$, we may either assume that $S$ is normal or that $p$ is a normalization morphism. In the first case, we may suppose that $S$ is connected, and replace $T$ by any of its connected components dominating $S$. The result then follows from Lemma \ref{lemnormal}.

We now deal with the second case where $p$ is a normalization morphism. 
Let $\cI\subset \cO_S$ be the annihilitor of the cokernel of $\cO_S\hookrightarrow p_*\cO_T$ (the \textit{conductor} of $p$). It is a coherent sheaf of $\cO_S$-ideals by \cite[Annex, \S 4.4]{GRCoherent} which is also, in view of its definition, a sheaf of $p_*\cO_T$-ideals. We let $S'\subset S$ and $T'\subset T$ be the (possibly nonreduced) complex subspaces defined by $\cI$.  By construction, the subset $S'$ of~$S$ is the locus over which $p$ is not an isomorphism, i.e.\ the nonnormal locus of $S$. As a consequence, one has ${\dim(S')<\dim(S)}$. We get a commutative exact diagram of coherent sheaves on $S$:
\begin{equation}
\begin{aligned}
\label{diagconductor}
\xymatrix
@R=0.5cm 
{
0 \ar[r]&\cI\ar[r]\ar@{=}[d]&p_*\cO_T\ar[r]&  p_*\cO_{T'} \ar[r] &0\\
0\ar[r]&\cI\ar[r]& \cO_S\ar[r]\ar[u]&\cO_{S'}\ar[u]\ar[r] &0.}
\end{aligned}
\end{equation} 

By (i) applied to the inclusion $T'\hookrightarrow T$, the subset $p^{-1}(S'\cap K)= T'\cap p^{-1}(K)$ of~$T'$ is $\cO(T')$-convex. The induction hypothesis applied to $p|_{T'}:T'\to S'$ then shows that $S'\cap K$ is $\cO(S')$\nobreakdash-convex. Use Lemma \ref{lemRunge} to choose a Runge neighborhood~$\Omega$
of $S'\cap K$ in $S'$ such that $s\notin \Omega$. Apply Lemma \ref{lemRunge} again to 
construct a relatively compact Runge neighborhood $\Theta$ of $p^{-1}(K)$ in $T$ that is disjoint from~$p^{-1}(\{s\})$, and such that $T'\cap\Theta\subset p^{-1}(\Omega)$.

The restriction map $\cO(\Theta)\to\cO(T'\cap\Theta)$ is continuous, and surjective by Cartan's Theorem B. It thus follows from the open mapping theorem for Fr\'echet spaces that there exists a neighborhood $U\subset \cO(T'\cap\Theta)$ of the origin such that for all $a\in U$, there exists $b\in \cO(\Theta)$ such that $b|_{T'\cap\Theta}=a$ and $|b|<1$ on $p^{-1}(K)$. Let $V\subset \cO(\Omega)$ be the inverse image of $U$ by the (continuous) pull-back map $\cO(\Omega)\to \cO(T'\cap\Theta)$.

Set $\cF:=\cO_{S'}$ if $s\notin S'$ and $\cF:=\cI_{\{s\}}\subset\cO_{S'}$ if $s\in S'$. By Proposition~\ref{OkaWeilnonreduit}~(iii), there exists $c\in H^0(S',\cF)\subset\cO(S')$ with $(1+c)|_{\Omega}\in V$. Set ${d:=1+p^*c\in\cO(T')}$ so that $a:=d|_{T'\cap\Theta}\in U$. 
By our choice of $U$, there exists $b\in\cO(\Theta)$ with ${b|_{T'\cap\Theta}=d|_{T'\cap\Theta}}$ and $|b|<1$ on $p^{-1}(K)$. The Oka--Weil approximation theorem \cite[Theorem~2.8.4]{Forstneric}
 now shows the existence of $e\in\cO(T)$ with $e|_{T'}=d$, and such that $|e|<1$ on $p^{-1}(K)$ and $e=1$ on $p^{-1}(\{s\})$.

Diagram (\ref{diagconductor}) remains exact after taking global sections by Cartan's Theorem~B.  A diagram chase in the resulting diagram shows the existence of $f\in\cO(S)$ with $p^*f=e\in\cO(T)$. One has $f(s)=1$ and $|f|<1$ on $K$, so $K$ is $\cO(S)$\nobreakdash-convex.
\end{proof}

\begin{rems}
(i) In Proposition \ref{propfinitehull} (ii), one cannot remove the assumption that~$S$ and~$T$ are Stein, even if $f$ is a reduction morphism (see \cite[(8.5)]{Schuster}) or a normalization morphism (see \cite[Theorem 3]{Markoe}).

(ii) In the setting of Proposition \ref{propfinitehull}, the equality $\widehat{p^{-1}(K)}_{\cO(S)}=p^{-1}(\whK_{\cO(T)})$ does not hold in general, for instance if $T=\{(x,y)\in\C^2\mid xy=0\}$, if $p:S\to T$ if the normalization morphism, and if $K=\{(x,0)\in\C^2\mid |x|=1\}\subset T$.
\end{rems}

\subsection{Stein compact subsets}

If $\cF$ is a sheaf on a complex space $S$ and $K\subset S$ is closed,  we let $\cF(K)$ denote the set of germs of sections of~$\cF$ in a neighborhood~of~$K$. 

A compact subset of a complex space~$S$ is said to be \textit{Stein} if it admits a basis of Stein open neighborhoods.
A \textit{Stein compactum} is the germ of a complex space along a Stein compact subset.
As the intersection of two Stein open subsets is Stein (see \cite[p.~127]{GRStein}), an intersection of Stein compact subsets is again Stein. By Lemma \ref{lemRunge}, an $\cO(S)$-convex compact subset of a Stein space is Stein. 

\begin{lem}
\label{flat}
Let $S$ be a Stein space and let $K\subset L\subset S$ be Stein compact subsets. Then the ring morphisms $\cO(S)\to\cO(K)$ and $\cO(L)\to \cO(K)$ are flat. 
\end{lem}

\begin{proof}
That $\cO(S)\to\cO(K)$ is flat is proven in \cite[Proof of Lemma~2.2]{Kucharz}. There, the compact subset $K$ is assumed to be $\cO(S)$-convex, but only the fact that it is Stein is used. In addition, the Stein space $S$ is assumed to be a connected manifold. The proof extends to our more general setting, replacing references to H\"ormander's book with applications of Cartan's Theorems A and B (in the generality of \cite{GRStein}).

Let $(U_i)_{i\in\N}$ be a decreasing basis of Stein open neighborhoods of $L$ in $S$.  By the above, the ring $\cO(K)$ is flat over $\cO(U_i)$ for all $i\in\N$. By the equational criterion of flatness \cite[Lemma \href{https://stacks.math.columbia.edu/tag/00HK}{00HK}]{SP}, the ring $\cO(K)$ is flat over $\colim_{i\in\N}(\cO(U_i))=\cO(L)$.
\end{proof}

We will say that a Stein compact subset $K$ of a Stein space $S$ is \textit{excellent} if the ring $\cO(K)$ is noetherian. By a theorem of Siu \cite[Theorem 1]{SiuNoeth}, this is the case if and only if for every germ $Z$ of closed analytic subset along~$K$, the subset $Z\cap K$ of $K$ has finitely many connected components. This holds in particular if~$K$~is semianalytic (see \cite[Th\'eor\`eme I.9]{Frisch}).
Excellent Stein compact subsets are plentiful as the next lemma shows.

\begin{lem}
\label{semianal}
Let $S$ be a Stein space and let $K\subset S$ be a compact subset.
\begin{enumerate}[(i)]
\item The subset $K$ has a semianalytic $\cO(S)$-convex compact neighborhood in $S$.
\item If $K$ is Stein, it has a basis of semianalytic Stein compact neighborhoods in~$S$.
\end{enumerate}
\end{lem}

\begin{proof}
After replacing it with $S^{\red}$, we may assume that $S$ is reduced.
Let $\rho:S\to\R$ be a real-analytic strongly psh exhaustion function (see \cite[Lemma p.~358]{NarasimhanLeviI}). For $c\gg 0$, the semianalytic compact subset $\{s\in S\mid \rho(s)\leq c\}$ is a neighborhood of $K$ which is $\cO(S)$-convex by Lemma \ref{critpsh} (ii), proving (i).

If $K$ is Stein, it admits a basis of Stein open neighborhoods each of which is exhausted by semianalytic Stein compact subsets by (i). This proves (ii).
\end{proof}

\begin{lem}
\label{neighneigh}
Let $K$ be a Stein compact subset of a Stein space $S$.  There exists a basis of Stein compact neighborhoods $(L_j)_{j\in J}$ of $K$ in $S$ such that $L_j$ admits a basis $(U_i)_{i\in I_j}$ of Stein open neighborhoods and $L_j$ is $\cO(U_i)$-convex for $j\in J$ and $i\in I_j$.
\end{lem}

\begin{proof}
Let $U$ be a Stein open neighborhood of $K$ in $S$.  By Lemma \ref{semianal} (i), there exists a compact neighborhood $L$ of $K$ in $U$ that is $\cO(U)$-convex.  By Lemma \ref{lemRunge}, the subset $L$ admits a basis of Runge open neighborhoods $(U_{i})_{i\in I}$ in $U$. It follows that $L$ is $\cO(U_i)$-convex for all $i\in I$. This concludes.
\end{proof}

If $K$ is a compact subset of Stein space $S$, we define $\cS_K\subset\cO(S)$ to be the subset of holomorphic functions that do not vanish on $K$, and $\cO(S)_K:=\cO(S)[\cS_K^{-1}]$. Such rings are studied in \cite{Kucharz} and \cite{ABF}, at least if $S$ is reduced.

\begin{lem}
\label{excellent}
Let $S$ be a Stein space and let $K\subset S$ be a compact subset. 
\begin{enumerate}[(i)]
\item If $K$ is $\cO(S)$-convex, the ring morphism $\cO(S)_K\to\cO(K)$ is faithfully flat.
\item The ring $\cO(S)_K$ is excellent.
\item If $K$ is Stein and excellent, the ring $\cO(K)$ is excellent.
\end{enumerate}
\end{lem}

\begin{proof}
By Lemma \ref{flat}, the ring~$\cO(K)$ is $\cO(S)$-flat, hence $\cO(S)_K$-flat. 
If $K$ is $\cO(S)$\nobreakdash-convex, the description given in \cite[Corollary~3.3~(v)]{ABF} of the maximal ideals of~$\cO(S)_K$ 
now implies assertion (i) (use \cite[Lemma \href{https://stacks.math.columbia.edu/tag/00HQ}{00HQ}]{SP}).

Assertion (ii) is proven in \cite[Theorem 3.5]{ABF}. 
As the noetherianity of~$\cO(S)_K$ is not explicitly proven there, we give an argument. If $K'$ is an excellent $\cO(S)$\nobreakdash-convex compact neighborhood of $K$ in $S$ (which exists by Lemma~\ref{semianal}~(i)), then $\cO(S)_K$ is a localization of $\cO(S)_{K'}$.  We may thus assume that $K$ is excellent and $\cO(S)$-convex. As $\cO(K)$ is noetherian, so is $\cO(S)_K$ by (i) and \cite[Lemma \href{https://stacks.math.columbia.edu/tag/033E}{033E}]{SP}.  To dispel any doubt as to whether the excellence of $\cO(S)_K$ holds if $S$ is possibly nonreduced, one can reduce to the reduced case using \cite[Main Theorem 1]{KS}.

Assertion (iii) is proven in \cite[Bemerkung pp.152-153]{Bingenerexcellent} (there, the Stein compact subset $K$ is assumed to be semianalytic but only its excellence is used).
\end{proof}

Let $K$ be a Stein compact subset in a Stein space $S$. By \cite[Theorem~4.7]{Zame}, for every prime ideal $\kp\subset\cO(K)$, there exist $s\in K$ and a germ~$Z$ of closed analytic subset along $K$ which is \textit{essentially irreducible at} $s$ in the sense of \cite[p.~118]{Zame}, and such that $\kp=\{f\in\cO(K)\mid f|_Z\textrm{ vanishes identically near }s\}$.

\begin{lem}
\label{regularlemma}
With the above notation, the following are equivalent.
\begin{enumerate}[(i)]
\item The ring $\cO(K)_{\kp}$ is regular.
\item The germ of $Z$ at $s$ is not included in the singular locus of $S$.
\end{enumerate}
\end{lem}

\begin{proof}
Assume first that $\cO(K)_{\kp}$ is regular.  Use Lemma \ref{semianal} (ii) to find an excellent Stein compact neighborhood $L$ of $K$ such that $Z$ extends to a germ of closed analytic subset along $L$. We still denote by $\kp$ the corresponding ideal of~$\cO(L)$.  As ${\cO(L)_{\kp}\to\cO(K)_{\kp}}$ is flat by Lemma \ref{flat},  the ring $\cO(L)_{\kp}$ is regular (see \cite[Lemma~\href{https://stacks.math.columbia.edu/tag/00OF}{00OF}]{SP}). As~$\cO(L)$ is excellent by Lemma \ref{excellent} (iii), the singular locus of~$\Spec(\cO(L))$ is closed, defined by an ideal $I\subset\cO(L)$. One can thus choose~$f\in I\setminus \kp$. As $f\notin\kp$ and $s\in\mathring{L}$, any neighborhood of $s$ in $S$ contains a point~$t\in Z\cap L$ with $f(t)\neq 0$. If $\km\subset\cO(L)$ is the ideal of functions vanishing at $t$, the ring $\cO(L)_{\km}$ is regular because $f\in I$. It is moreover of dimension $d:=\dim_t(Z)$ (see \cite[Proof of Corollary 4.9]{Zame}). As $\km/\km^2\isoto\km_{S,t}/\km_{S,t}^2$ (see \cite[below Lemma 2]{Kiehl}),
the ring~$\cO_{S,t}$ is regular, so $S$ is nonsingular at $t$ (see \cite[6, \S~2.1]{GRCoherent}).

Assume now that (ii) holds. Let $Z'$ be an irreducible component of the germ of~$Z$ at $s$ which is not included in $S^{\sing}$. If $\kp'\subset \cO_{S,s}$ is the  ideal of functions vanishing on $Z'$, then $(\cO_{S,s})_{\kp'}$ is regular by \cite[Th\'eor\`eme 3]{HouzelIII}. 
Since $\cO(K)_{\kp}\to(\cO_{S,s})_{\kp'}$ is a flat local morphism of local noetherian rings (by Lemma \ref{flat} and \cite[Proposition 4.11]{Zame}), the ring $\cO(K)_{\kp}$ is regular by \cite[Lemma \href{https://stacks.math.columbia.edu/tag/00OF}{00OF}]{SP}.
\end{proof}

Recall that a ring morphism $A\to B$ is \textit{regular} if it is flat and the induced morphism $\Spec(B)\to\Spec(A)$ has locally noetherian and geometrically regular fibers (see \cite[Definition \href{https://stacks.math.columbia.edu/tag/07R7}{07R7}]{SP}).
Lemma \ref{regular} is used in the proof of Lemma~\ref{cdbK}.

\begin{lem}
\label{regular}
Let $S$ be a Stein space and let $K\subset L\subset S$ be excellent Stein compact subsets. 
Then the ring morphisms $\cO(L)\to\cO(K)$ and $\cO(S)_K\to\cO(K)$ are regular. 
\end{lem}

\begin{proof}
The first assertion is proven in \cite[(2.2)]{Bingener} (there, Lemmas \ref{flat} and \ref{regularlemma} are used as well-known facts, and~$K$ and $L$ are assumed to be semianalytic but only their excellence is used).  
We prove the second assertion following the arguments of \loccit As~$\cO(K)$ is $\cO(S)_K$-flat by Lemma \ref{flat}, and in view of \cite[Lemma \href{https://stacks.math.columbia.edu/tag/038V}{038V} (1)$\Leftrightarrow$(3)]{SP}, it suffices to show that the fibers of $\Spec(\cO(K))\to \Spec(\cO(S)_K)$ are regular. Fix $\kp\in\Spec(\cO(K))$ and let $\kq$ be its image in $\Spec(\cO(S))$. 
We will show the regularity of $\cO(K)/\kq\,\cO(K)$ at~$\kp$.
Set $T:=\{s\in S\mid f(s)=0\textrm{ for all }f\in\kq\}$.

The restriction morphism $\cO(K)/\kq\,\cO(K)\to\cO(T\cap K)$ is surjective by Cartan's Theorem B. It is also injective, as for any relatively compact Stein open neighborhood $U$ of $K$ in~$S$, one can find $(f_i)_{1\leq i\leq k}$ in $\kq$ such that $\cO_U^{\oplus k}\xrightarrow{(f_i)}\cI_T|_{U}$ is onto, and apply Cartan's Theorem B to show that $\cI_T(K)=\kq\,\cO(K)$. After replacing $S$ with $T$ and $K$ with $T\cap K$, we may thus assume that $\kq=0$ (and $S$ is reduced).

Let $s\in K$ and $Z$ be as in Lemma \ref{regularlemma}. As $\kq=0$, no nonzero element of $\cO(S)$ vanishes on the germ of $Z$ at $s$. As $S$ is reduced, $S^{\sing}$ has a nonzero equation in~$\cO(S)$, and hence the germ of $Z$ at $s$ is not included in $S^{\sing}$. By Lemma \ref{regularlemma}, the ring $\cO(K)_{\kp}$ is regular, as wanted.
\end{proof}

\subsection{Meromorphic functions and ramified coverings}

A finite holomorphic map $p:T\to S$ between reduced complex spaces is an \textit{analytic covering}
if there exists a nowhere dense closed analytic subset $\Sigma\subset S$ such that $p^{-1}(\Sigma)$ is nowhere dense in $T$ and ${p|_{p^{-1}(S\setminus \Sigma)}:p^{-1}(S\setminus \Sigma)\to S\setminus \Sigma}$ is a local biholomorphism (unlike in \cite[7, \S 2.1]{GRCoherent},  we do not insist that~$p$ be surjective).  It is said to be \textit{of degree}~$d$ if the fibers of~$p|_{p^{-1}(S\setminus \Sigma)}$ have cardinality $d$.

We let $\cM$ denote the sheaf of germs of meromorphic functions on a complex space (see \cite[6, \S 3.1]{GRCoherent}). 
If~$S$ is a reduced and irreducible Stein space, then~$\cM(S)$ is the fraction field of the domain~$\cO(S)$ (because the sheaf of denominators of $h\in\cM(S)$ defined in \cite[6, \S 3.2]{GRCoherent} is coherent and nonzero, hence admits a nonzero global section).
The following proposition is classical in dimension $1$, see \eg \cite[1, \S 4.14, Corollaries 4 and~5]{ShokuRiemann}.

\begin{prop}
\label{eqcat}
Let $S$ be a reduced and irreducible Stein space. Then the functor 
$$
 \left\{  \begin{array}{l}
    \textrm{\hspace{.2em}analytic coverings }p:T\to S\\\textrm{\hspace{0em}with $T$ connected and normal}
  \end{array}\right\}\to
 \left\{  \begin{array}{l}
    \textrm{finite field extensions }\\\hspace{2.2em}\cM(S)\subset F
  \end{array}\right\}
$$
which associates with $p:T\to S$ the extension $\cM(S)\subset\cM(T)$ of meromorphic function fields is an equivalence of categories.
\end{prop}

\begin{proof}
After replacing $S$ with its normalization (which is legitimate by \cite[8, \S 1.3, Lifting Lemma]{GRCoherent}), we may assume that $S$ is connected and normal.

We first show that the functor is well-defined. Let $S^{\sing}$ and $T^{\sing}$ be the singular loci of $S$ and $T$.
Then $Z:=S^{\sing}\cup p(T^{\sing})$ is a closed analytic subset of~$S$ which has codimension $\geq 2$ by  \cite[6, \S 5.3]{GRCoherent}. Set $U:=S\setminus Z$. The locus $W\subset T\setminus p^{-1}(Z)$ where $p$ is not a local biholomorphism is an analytic subset of pure codimension $1$ because it is the zero locus of the Jabobian of $p$.  Over $U\setminus p(W)$, the map $p$ is a finite local biholomorphism. As $U\setminus p(W)$ is connected by \cite[9, \S1.2]{GRCoherent}, the map $p|_{p^{-1}(U)}:p^{-1}(U)\to U$ is a degree $d$ analytic covering of complex manifolds for some $d\geq 1$. It then follows from \cite[7, \S 3.1, Corollary~2]{GRCoherent} that the field extension $\cM(U)\subset\cM(p^{-1}(U))$ is finite (of degree~$\leq d$). As $\cM(S)=\cM(U)$ and $\cM(T)=\cM(p^{-1}(U))$ by \cite[9, \S 5.2]{GRCoherent}, the field extension $\cM(S)\subset\cM(T)$ is finite (of degree~$\leq d$).

Let $p_1:T_1\to S$ and $p_2:T_2\to S$ be analytic coverings of connected normal Stein spaces.  By \cite[Theorem II]{Isssa}, any morphism $\cM(T_1)\subset \cM(T_2)$ of $\C$-algebras is induced by a unique holomorphic map $f:T_2\to T_1$,  and the former is a morphism of $\cM(S)$-algebras if and only if $p_2=p_1\circ f$. The functor is therefore full and faithful.

We finally prove that the functor is essentially surjective.  Let $F$ be a finite extension of degree $d$ of $\cM(S)$. Let $P(x)=x^d+\sum_{i=0}^{d-1} a_ix^i\in\cM(S)[x]$ be an irreducible polynomial with splitting field $F$.
Let $b\in\cO(S)$ be a nonzero element with $ba_i\in\cO(S)$ for $i\in\{0,\dots, d-1\}$, and let $R\subset\P^1(\C)\times S$ be the zero locus of $b X^d+\sum_{i=0}^{d-1} ba_iX^iY^{d-i}$,
where $[X:Y]$ are homogeneous coordinates of~$\P^1(\C)$. Let~$R'$ be an irreducible component of $R$  dominating $S$, let $R'\to R''\to S$ be the Stein factorization of the projection $R'\to S$ (see \cite[10, \S 6.1]{GRCoherent}), and let~$T$ be the normalization of $R''$, with projection $p:T\to S$. As the inverse image by~$R\to S$ of a general point of $S$ has cardinality $\leq d$, the same holds for the inverse image by $p$ of a general point of $S$ (note that the normalization morphism~$T\to R''$ is an isomorphism above the complement of a discrete subset of $R''$).  We deduce from the first paragraph of this proof that the field extension ${\cM(S)\subset\cM(T)}$ has degree~$\leq d$. As the element $X/Y\in\cM(T)$ is annihilated by $P$, we get inclusions ${\cM(S)\subset F\subset\cM(T)}$. A degree argument now shows that $F$ and~$\cM(T)$ are isomorphic extensions (of degree~$d$) of~$\cM(S)$.
\end{proof}

\begin{rem}
It follows from the proof of Proposition \ref{eqcat} that a finite extension $\cM(S)\subset F$ of degree $d$ corresponds to an analytic covering $p:T\to S$ of degree $d$.
\end{rem}

We deduce at once from Proposition \ref{eqcat} the following corollary.

\begin{cor}
\label{eqcatetale}
Let $S$ be a reduced Stein space with finitely many irreducible components. 
Associating with $p:T\to S$ the $\cM(S)$-algebra $\cM(T)$ induces an equivalence of categories
$$
 \left\{  \begin{array}{l}
    \textrm{\hspace{0em}analytic coverings }p:T\to S\\\hspace{3em}\textrm{with $T$ normal}
  \end{array}\right\}\to
 \left\{  \begin{array}{l}
    \textrm{\hspace{.9em}finite \'etale}\\\cM(S)\textrm{-algebras}
  \end{array}\right\}.
$$
\end{cor}

The next lemma will be used in the proofs of Propositions \ref{finitecompact} (ii) and~\ref{killloc}.

\begin{lem}
\label{Galoisclosure}
Let $q:R\to S$ be an analytic covering of connected normal Stein spaces. There exists an analytic covering $p':T\to R$ of connected normal Stein spaces and a finite group $\Gamma$ acting on $T$ and acting trivially on~$S$, such that $p:=q\circ p'$ is $\Gamma$-equivariant and $\Gamma$ acts transitively on the fibers of~$p$. 
\end{lem}

\begin{proof}
The field extension $\cM(S)\subset \cM(R)$ is finite by Proposition \ref{eqcat}. Let $\cM(S)\subset F$ be its Galois closure, with Galois group $\Gamma$. Let $p':T\to R$ be the analytic covering associated with $\cM(R)\subset F$ through the equivalence of categories of Proposition~\ref{eqcat}. By functoriality, the finite group~$\Gamma$ acts on $T$ and $p:=q\circ p'$ is $\Gamma$\nobreakdash-equivariant. The quotient complex space $T/\Gamma$ (see \cite[Theorem~4]{Cartan}) is connected, normal and admits a natural finite holomorphic map $T/\Gamma\to S$.
One deduces from the field inclusions $\cM(S)\subset \cM(T/\Gamma)\subset\cM(T)^{\Gamma}=\cM(S)$ that $\cM(S)=\cM(T/\Gamma)$. In view of Proposition~\ref{eqcat}, the projection $T/\Gamma\to S$ is a biholomorphism. Consequently, the group $\Gamma$ acts transitively on each fiber of~$p$.
\end{proof}

\section{Killing cohomology classes on finite coverings}
\label{parunr}

The goal of this section is Proposition \ref{killunram}, proven in \S\ref{paru3} by induction on the degree of a cohomology class. We cover the base case of the induction in \S\ref{paru1} and the induction procedure relies on Grauert's bump method described in~\S\ref{paru2}.

\subsection{Finite coverings of \texorpdfstring{$\cO(S)$}{O(S)}-convex compact subsets}
\label{paru1}

 The following proposition will be key in dealing with degree $1$ cohomology classes.

\begin{prop}
\label{finitecompact}
Let $S$ be a Stein manifold of dimension $n$, let $K\subset S$ be an $\cO(S)$\nobreakdash-convex compact subset, 
let $U\subset S$ be an open neighborhood of $K$ and let ${f:\widetilde{U}\to U}$ be a finite surjective local biholomorphism.
\begin{enumerate}[(i)]
\item After maybe shrinking $U$, there exist a finite surjective holomorphic map ${q:R\to S}$ and an open embedding $h:\wU\hookrightarrow R$ such that $q\circ h=f$.
\item After maybe shrinking $U$,  there exist a finite surjective holomorphic map ${p:T\to S}$ and a holomorphic map $g:p^{-1}(U)\to\wU$ with $f\circ g=p|_{p^{-1}(U)}$.
\end{enumerate}
\end{prop}

\begin{proof}[Proof of (i)]
By Lemma \ref{lemRunge},  we may assume that $U$ is Stein after shrinking it, hence that so is $\wU$ by \cite[V, \S 1, Theorem 1 d)]{GRStein}. 
By \cite[Theorem 3]{Narasimhan}, one can then find an embedding $i:\wU\hookrightarrow \C^N$ (with $N=2n+1$). Consider the embedding $j:\wU\hookrightarrow \C^N\times U$ defined by $j(z)=(i(z),f(z))$.
Write $\pr_1:\C^N\times U\to\C^N$ and $\pr_2:\C^N\times U\to U$ for the projections.
For any continuous $\varepsilon:U\to ]0,+\infty[$, define
\begin{equation}
\label{tubul}
\Theta_{\varepsilon}:=\{(a,b)\in \C^N\times U\mid\exists\,c\in\wU \textrm{ with } f(c)=b \textrm{ and }|a-i(c)|<\varepsilon(b)\}.
\end{equation}
Choose $\varepsilon$ sufficiently small so that the element $c$ in (\ref{tubul}) is always unique and define the map~$\pi:\Theta_{\varepsilon}\to \wU$ by the property that $\pi(a,b)=c$.
Geometrically, the subset~$\Theta_{\varepsilon}$ of $\C^N\times U$ is a tubular neighborhood of $j(\wU)$ in the sense that the retraction $j\circ \pi:\Theta_{\varepsilon}\to j(\wU)$ of the inclusion map is a disc bundle.
By \cite[Proposition~2.1]{CM}, there exists a neighborhood~$\Omega$ of $j(\wU)$ in $\Theta_{\varepsilon}$ that is Runge in~$\C^N\times U$.
Finally, choose $\delta:U \to]0,+\infty[$ continuous sufficiently small so that $\overline{\Theta_{\delta}}\subset \Omega$.

Define a holomorphic map $\phi:\Omega\to \C^N$ by setting $\phi(z)=\pr_1(z)-i\circ\pi(z)$.  The map $\phi$ is submersive along $\overline{\Theta_{\delta}}$,  and its zero locus in $\overline{\Theta_{\delta}}$ is included in $\Theta_{\delta}$ and projects biholomorphically to $\wU$ by $\pi$.
All these properties persist (after maybe shrinking~$U$) if one replaces $\phi$ with a holomorphic map $\phi':\Omega\to\C^N$ close enough to $\phi$ on $\overline{\Theta_{\delta}}\cap \pr_2^{-1}(K)$ in the $C^0$-topology (hence in the $C^1$-topology by the Cauchy estimates).
We construct such a $\phi'$ as follows. First approximate $\phi$ by the restriction of a holomorphic map $\C^N\times U\to\C^N$ using that~$\Omega$ is Runge in $\C^N\times U$. Expand this map as $N$ power series in $N$ variables with coefficients in $\cO(U)$, truncate these power series to turn them into~$N$ polynomials in $N$ variables with coefficients in~$\cO(U)$,  and then use the Oka--Weil approximation theorem \cite[VII, A, Theorem 6]{GunningRossi} to approximate these coefficients on~$K$ by restrictions of elements of $\cO(S)$.  Let $\psi:\C^N\times S\to\C^N$ be the holomorphic map (polynomial in the first variable) obtained in this way and set $\phi':=\psi|_{\Omega}$.

Let $\Psi_1,\dots,\Psi_N\in\cO(S)[X_0,\dots,X_N]$ be the homogenizations of the components of $\psi$, and let $\whR\subset\P^N(\C)\times S$ be the locus where they all vanish, with projection $\hq:\whR\to S$. The map $\hq$ is surjective by \cite[I, Theorem 7.2]{Hartshorne}. Hence so is the finite holomorphic map $q:R\to S$ appearing in the Stein factorization \cite[10,~\S 6.1]{GRCoherent} of $\hq$. By our choice of $\phi'$, the intersection of $\whR$ with $\Theta_{\delta}$ is biholomorphic (via the map $\pi$) to $\wU$. The resulting open embedding $\hh:\wU\to \whR$ satisfies $\hq\circ\hh=f$. 

The uniqueness of the Stein factorization shows that it is compatible with base change by open embeddings. As $\hh(\wU)$ is a connected component of $\hq^{-1}(U)$ such that $\hq|_{\hh(\wU)}:\hh(\wU)\to U$ is finite (as $f$ is), we deduce that the composition of $\hh$ with the natural map $\whR\to R$ is an open embedding $h:\wU\to R$ such that $q\circ h=f$.
\end{proof}

\begin{proof}[Proof of (ii)]
Let $q:R\to S$,  and $h:\wU\hookrightarrow Y$ be as in (i). We may assume that~$R$ is normal after normalizing it, and that $S$ is connected after replacing it with one of its connected components. After shrinking $U$, we may suppose that it has finitely many connected components $U_1,\dots,U_r$. If $p_i:T_i\to S$ solves our problem for $f|_{f^{-1}(U_i)}:f^{-1}(U_i)\to U_i$ instead of $f$, then we may choose $T$ to be the fiber product $T_1\times_S\dots\times_S T_r$. We may thus assume that~$U$ is connected. Replacing~$\wU$ with one of its connected components, we may also suppose that~$\wU$ is connected. Finally replacing $R$ with its connected component containing~$h(\wU)$, we reduce to the case where $R$ is connected.

Let $p':T\to R$ and $\Gamma$ be as in Lemma~\ref{Galoisclosure}, and set $p:=q\circ p'$. Let $V$ be a connected component of~$p^{-1}(U)$.  By the transitivity of the action of $\Gamma$ on the fibers of $p$, one may choose $\gamma_V\in\Gamma$ such that $\gamma_V(V)\subset p'^{-1}(\wU)$. One then defines $g(x):=p'\circ\gamma_V(x)$ for $x\in V$.
\end{proof}

\subsection{Grauert's bump method}
\label{paru2}

To prove Proposition \ref{killunram}, we will use Grauert's bump method as developed by Henkin and Leiterer \cite{HL}. 
In Proposition \ref{bump} below, we sum up what we exactly need, relying on the exposition by Forstneri\v{c}~\cite{Forstneric}.

\begin{lem}
\label{coveringballs}
Let $S$ be a Stein manifold of dimension $n$. Any point $s\in S$ admits an $\cO(S)$-convex compact neighborhood $B\subset S$ such that $B$ admits a basis of contractible open neighborhoods in $S$.
\end{lem}

\begin{proof}
By \cite[Theorem 3]{Narasimhan}, one may view $S$ as a submanifold of $\C^N$ for $N\gg 0$. Then $\rho:S\to \R$ defined by $\rho(x):=|x-s|^2$ is a $\ci$ strongly psh exhaustion function, and $B:=\{\rho\leq\varepsilon\}$ works for $0<\varepsilon\ll 1$ (see Lemma~\ref{critpsh} (ii)).
\end{proof}

\begin{prop}
\label{bump}
Let $S$ be a Stein manifold of dimension $n$. Then there exist sequences $(A_0,A_1,\dots)$ and $(B_0,B_1,\dots)$ of compact subsets of $S$ such that:
\begin{enumerate}[(i)]
\item the subsets $A_i$ and $B_i$ are $\cO(S)$-convex;
\item the subset $B_i$ has a basis of contractible open neighborhoods;
\item one has $A_0=\varnothing$ and  $A_{i+1}\subset A_i\cup B_i$;
\item for any compact subset $K\subset S$, one has $K\subset A_i$ for $i\gg 0$.
\end{enumerate}
\end{prop}

\begin{proof}
Let $\rho:S\to \R$ be a $\ci$ strongly psh exhaustion function (see \cite[Lemma~p.~358]{NarasimhanLeviI}).  By \cite[Proposition 6.13]{GG} and \cite[Lemma 3.10.3]{Forstneric}, we may assume that its critical points are nondegenerate,  have distinct images by~$\rho$, and 
are nice in the sense of \cite[Definition 3.10.2]{Forstneric}.

We adopt the following nonstandard terminology. We say that a pair $(A,A')$ of $\cO(S)$-convex compact subsets of $S$ is a \textit{bump} if there exists an $\cO(S)$-convex compact subset $B\subset S$ with a basis of contractible neighborhoods with $A'\subset A\cup B$. The proof has two parts, which, combined,  prove the proposition.

 In the first part (the \textit{noncritical extension case}), we show that if ${c<c'}$ are such that $\rho$ has no critical values in $[c,c']$, then one can go from $A:=\{\rho\leq c\}$ to $A':=\{\rho\leq c'\}$ by a sequence of bumps containing $A$. This follows at once from \cite[Proof of Lemma 5.10.3]{Forstneric}. 
Indeed, it is explained there how to go from $A$ to $A'$ by an increasing sequence of sublevel sets of (varying)~$\ci$ psh exhaustion functions (which are $\cO(S)$-convex by Lemma \ref{critpsh} (ii)),
 in a way that the difference between two successive ones is included in an open subset of any fixed open covering of~$A'$. In view of Lemma \ref{coveringballs}, one can ensure that this difference is included in an $\cO(S)$\nobreakdash-convex compact subset with a basis of contractible neighborhoods.  

In the second part (the \textit{critical extension case}), we fix a critical value of $\rho$ (assumed without loss of generality to be equal to $0$),  with associated critical point $p_0\in S$. We show that there exist $c<0<c'$ such that one can go from $A:=\{\rho\leq c\}$ to $A':=\{\rho\leq c'\}$ by a sequence of bumps containing $A$. 
To do so, use Lemma \ref{coveringballs} to choose a compact $\cO(S)$-convex neighborhood of $p_0$ in~$S$ with a basis of contractible open neighborhoods.  Apply the constructions of \cite[p.~102]{Forstneric} with $q=1$ and  $U\subset B$ an open neighborhood of $p_0$. In particular, choose~$c_0>0$ small enough so that \cite[Lemma 3.11.4]{Forstneric} applies, and let $0<t_0<t_1$ and $\tau:\{\rho< 3c_0\}\to\R$  be as in \loccit Set $c:=-t_0$ and $c':=c_0$.  One can go from $A$ to $\{\tau\leq t_1-t_0\}$ by a bump by \cite[Lemma 3.11.4 (c)]{Forstneric}.  Arguing as in the noncritical extension case (using the~$\ci$ strongly psh function~$\tau$ on $\Omega:=\{\rho<3c_0\}$, which is legitimate by \cite[Proposition 3.11.4 (d)]{Forstneric}), one sees that one can go from $\{\tau\leq t_1-t_0\}$ to $\{\tau\leq 2c_0 \}$ by a sequence of bumps (the involved compact $\cO(\Omega)$-convex sets are also $\cO(S)$-convex because $\Omega$ is Runge in $S$ by Lemma \ref{critpsh} (i)).  Finally, one can go from $\{\tau\leq 2c_0\}$ to $A'$ by a bump by \cite[Lemma~3.11.4~(b)]{Forstneric}.
\end{proof}

\subsection{Induction on the cohomological degree}
\label{paru3}

We finally reach our goal.

\begin{prop}
\label{killunram}
Let $S$ be a Stein space, let $K\subset S$ be an $\cO(S)$-convex compact subset, 
and let $U\subset S$ be an open neighborhood of $K$. Fix integers $k, m\geq 1$. For all $\alpha\in H^k(U,\Z/m)$, there exists a finite surjective holomorphic map $p:T\to S$ and an open neighborhood $V$ of $K$ in $U$ such that $\alpha|_{p^{-1}(V)}=0$ in $H^k(p^{-1}(V),\Z/m)$.
\end{prop}

In the statement of Proposition \ref{killunram}, we use the notation $\alpha|_{p^{-1}(V)}$ to denote the pull-back of $\alpha$ by the map $p|_{p^{-1}(V)}:p^{-1}(V)\to U$.

\begin{proof}
We argue by induction on $k\geq 1$, which we fix.
Using Proposition \ref{propfinitehull} (i), we may replace $S$ with an irreducible component of its reduction, and assume that~$S$ is irreducible,  hence of finite dimension $n$.  We further reduce to the case where~$S$ is a manifold as follows. Let $i:S\to S'$ be as in Lemma \ref{lemretract}, fix a continuous retraction $r:S'\to i(S)$ and replace $S$, $K$, $U$ and $\alpha$ by $S'$, $i(K)$, $r^{-1}(U)$ and~$r^*\alpha$, noting that~$i(K)$ is $\cO(i(S))$-convex by Proposition \ref{propfinitehull} (ii), hence $\cO(S')$-convex.  This is legitimate because, if the proposition is proved for $r^*\alpha$ using a finite surjective map $p':T'\to S'$, then it is also proved for $i^*r^*\alpha=\alpha$ using the finite surjective map~$p:T'\times_{S'}S\to S$ obtained from $p'$ by base change by $i:S\to S'$.

Assume first that $k=1$. Let $f:\widetilde{U}\to U$ be the degree $m$ unramified cyclic covering associated with $\alpha$. After maybe shrinking $U$, choose $p:T\to S$ and $g:p^{-1}(U)\to\wU$ with $f\circ g=p|_{p^{-1}(U)}$ as in Proposition \ref{finitecompact} (ii). As $f^*\alpha=0$ by choice of $f$, one has $\alpha|_{p^{-1}(U)}=g^*f^*\alpha=0$, as wanted.

Assume now that $k\geq 2$.  After shrinking $U$, we may assume that it is a Runge domain in $S$ (see Lemma \ref{lemRunge}). Let $(A_0,A_1,\dots)$ and $(B_0,B_1,\dots)$ be as in Proposition~\ref{bump} applied to the Stein manifold $U$.  As $U$ is Runge in $X$, it follows from Proposition~\ref{bump} (i) that the $A_i$ and the $B_i$ are $\cO(S)$-convex compact subsets of $S$. 

We will now show by induction on $i\geq 0$ that there exists an open neighborhood~$V_i$ of~$A_i$ in $U$ and a finite surjective holomorphic map $p_i:T_i\to S$ such that $\alpha|_{p_i^{-1}(V_i)}=0$ in $H^k(p_i^{-1}(V_i),\Z/m)$.  When $i=0$, one can take $V_0=\varnothing$. 

Assume that we have proven the statement for $i$ and let us prove it for $i+1$. Let~$W_i$ be a contractible open neighborhood of $B_i$ in $U$ (see Proposition \ref{bump} (ii)).  Consider the boundary map 
$$\partial: H^{k-1}(p_i^{-1}(V_i\cap W_i),\Z/m)\to H^{k}(p_i^{-1}(V_i\cup W_i),\Z/m)$$
in the Mayer--Vietoris exact sequence.  As $\alpha|_{p_i^{-1}(V_i)}=0$ by the induction hypothesis on $i$ and $\alpha|_{p_i^{-1}(W_i)}=(p_i|_{p_i^{-1}(W_i)})^*(\alpha|_{W_i})=0$ because $W_i$ is contractible, there exists $\beta\in H^{k-1}(p_i^{-1}(V_i\cap W_i),\Z/m)$ such that $\partial(\beta)=\alpha|_{p_i^{-1}(V_i\cup W_i)}$. By the induction hypothesis on the cohomological degree (applied in degree $k-1$ to the $\cO(T_i)$\nobreakdash-convex compact subset $p_i^{-1}(A_i\cap B_i)$ of $T_i$, see Proposition \ref{propfinitehull} (i)),  there exist a finite surjective holomorphic map $q_i:T_{i+1}\to T_{i}$ and 
an open neighborhood $\Omega_i$ of~${p_i^{-1}(A_i\cap B_i)}$ in~$p_i^{-1}(V_i\cap W_i)$ such that $\beta|_{q_i^{-1}(\Omega_i)}=0$. Let $V_i'$ and~$W_i'$ be open neighborhoods of~$A_i$ and $B_i$ in~$V_i$ and~$W_i$ respectively, such that $p_i^{-1}(V'_i\cap W'_i)\subset\Omega_i$.  Set $p_{i+1}:=p_i\circ q_i$ and $V_{i+1}:=V_i'\cup W_i'$. Then $\alpha|_{p_{i+1}^{-1}(V_{i+1})}$ is the image of $\beta|_{p_{i+1}^{-1}(V'_i\cap W'_i)}$ by the boundary map
$$\partial': H^{k-1}(p_{i+1}^{-1}(V'_i\cap W'_i),\Z/m)\to H^{k}(p_{i+1}^{-1}(V'_i\cup W'_i),\Z/m)$$
of the Mayer--Vietoris exact sequence, by compatibility of $\partial$ and $\partial '$.  As $\beta|_{p_{i+1}^{-1}(V'_i\cap W'_i)}$ vanishes (because $\beta|_{q_i^{-1}(\Omega_i)}=0$), it follows that $\alpha|_{p_{i+1}^{-1}(V_{i+1})}=0$, as required.

In view of Proposition \ref{bump} (iv),  one can take $V=V_i$ and $p=p_i$ for $i\gg0$. 
\end{proof}

\section{A relative comparison theorem}
\label{parcomprel}

After studying analytifications of algebraic varieties over Stein algebras in~\S\ref{parBingener}, and \'etale sheaves on them in \S\ref{paretale}, we prove the relative Artin comparison theorem in Stein geometry in \S\ref{parrelArtin} (Theorem~\ref{relArtin}).

\subsection{Analytification of algebraic varieties over Stein algebras}
\label{parBingener}


Fix a Stein space $S$. Beware that the ring $\cO(S)$ is in general not noetherian.
Let~$(P)$ be a property of morphisms of schemes. A morphism $f:X\to Y$ of $\cO(S)$\nobreakdash-schemes is said to \textit{interiorly} satisfy~$(P)$ if there exists an open covering $(U_i)_{i\in I}$ of~$S$ such that the morphisms $f_{\cO(U_i)}:X_{\cO(U_i)}\to Y_{\cO(U_i)}$ satisfy $(P)$. We say that an $\cO(S)$-scheme~$X$ interiorly satisfies $(P)$ if so does the structural morphism $X\to \Spec(\cO(S))$.

\begin{lem}
\label{int}
Let $(P)$ be a property of morphisms of schemes that is stable by base change and fpqc local on the base.
Let $S$ be a Stein space and let $f:X\to Y$ be a morphism of $\cO(S)$-schemes that interiorly satisfies $(P)$.
\begin{enumerate}[(i)]
\item For $K\subset S$ compact, $f_{\cO(K)}:X_{\cO(K)}\to Y_{\cO(K)}$ satisfies $(P)$. 
\item For $K\subset S$ compact, $f_{\cO(S)_K}:X_{\cO(S)_K}\to Y_{\cO(S)_K}$ satisfies $(P)$. 
\item For $U\subset S$ open and relatively compact, $f_{\cO(U)}:X_{\cO(U)}\to Y_{\cO(U)}$ satisfies $(P)$. 
\end{enumerate}
\end{lem}

\begin{proof}
To prove (i) and (ii), we may replace $K$ with $\whK_{\cO(S)}$ (by base change), and consequently assume that $K$ is $\cO(S)$-convex, hence Stein.  
Any $s\in S$ has a basis of Stein compact neighborhoods (see Lemma \ref{semianal} (ii)). By hypothesis on~$f$ and the base change property,  each $s\in K$ has a Stein compact neighborhood $K_s$ such that $f_{\cO(K_s)}$ satisfies $(P)$.  By compactness of~$K$, we may extract a finite family $(K_i)_{1\leq i\leq k}$ covering $K$. Replacing~$K_i$ with $K_i\cap K$ (and using the base change property again), we may assume that $K_i\subset K$ for $1\leq i\leq k$. The morphism $\cO(K)\to\prod_{i=1}^k\cO(K_i)$ is flat by Lemma~\ref{flat}, hence faithfully flat by \cite[Lemma \href{https://stacks.math.columbia.edu/tag/00HQ}{00HQ}]{SP} and the description of the maximal ideals of $\cO(K)$ given in \cite[Corollary 3.3]{Zame}. Assertion~(i) follows since~$(P)$ is fpqc local on the base. So does (ii) by Lemma~\ref{excellent}~(i).
Assertion (iii) follows from (i) applied with $K=\overline{U}$ (by base change).  
\end{proof}

An $\cO(S)$-scheme interiorly locally of finite type is interiorly locally of finite presentation (use Lemmas \ref{semianal} (ii) and \ref{int} (i)), 
\ie belongs to the category~$\mathbb{K}$ in the terminology of Bingener \cite[p.~23]{Bingener}.
With such an $\cO(S)$\nobreakdash-scheme $X$,  Bingener associates its analytification $X^{\an}$ (see \cite[Satz 1.1]{Bingener} when $X$ is of finite presentation over~$\cO(S)$; the construction and its properties extend under our more general hypotheses as indicated in \cite[p.~23]{Bingener} as one sees by applying them to the $\cO(U)$-schemes $X_{\cO(U)}$ for relatively compact Stein open subsets $U\subset S$). 

The analytification $X^{\an}$ of $X$ is a (possibly nonseparated)
complex space over~$S$ endowed with a map ${i_X:X^{\an}\to X}$ of locally ringed spaces, characterized by the fact that the map
\begin{alignat*}{4}
\Hom_S(T,X^{\an})&\to \Hom_{\cO(S)}(T,X)\\
h\hspace{1em}&\mapsto \hspace{1em} i_X\circ h.
\end{alignat*}
is bijective for all complex spaces $T$ over $S$.
In particular, with each point $x\in X^{\an}$ corresponds a closed point of $X$ with complex residue field, also denoted by $x$.

When $X$ is an affine $\cO(S)$-scheme of finite presentation, with coordinate ring $\cO(S)[x_1,\dots,x_N]/\langle f_1,\dots, f_M\rangle$, the analytification $X^{\an}$ of $X$ can be described concretely as the zero locus in $S\times \C^N$ of $f_1,\dots,f_M\in\cO(S\times\C^N)$.
The analytification of an $\cO(S)$-scheme locally of finite presentation can then be constructed by gluing the analytifications of affine open subsets covering $X$.  In the general interiorly locally of finite type case, one glues the analytifications of $\cO(U_i)$-schemes of the form $X_{\cO(U_i)}$, where 
$(U_i)_{i\in I}$ is a cover of $S$ by relatively compact Stein open subsets.
Concrete examples of analytifications  include $\Spec(\cO(S))^{\an}=S$, $(\A^N_{\cO(S)})^{\an}=S\times \C^N$ and $(\P^N_{\cO(S)})^{\an}=S\times \P^N(\C)$.

This construction is functorial \cite[p.~2]{Bingener} (and we let $f^{\an}:X^{\an}\to Y^{\an}$ denote the holomorphic map induced by a morphism $f:X\to Y$ of $\cO(S)$-schemes interiorly locally of finite type), commutes with the formation of fiber products \cite[p.~3]{Bingener}, and is compatible with change of the base Stein space \cite[(1.2)]{Bingener}. The next proposition is a consequence of \cite[Satz 3.1]{Bingener}. 

\begin{prop}
\label{permanencean}
Let $S$ be a Stein space. Let $f:X\to Y$ be a morphism interiorly of finite type between $\cO(S)$-schemes interiorly locally of finite type.
If $f$ is interiorly surjective (\resp separated, proper, finite, flat, \'etale, a closed embedding, an open embedding), then $f^{\an}$ is surjective (\resp separated, proper, finite, flat, a local biholomorphism, a closed embedding, an open emebdding).
\end{prop}

We also state the following proposition for later reference.

\begin{prop}
\label{extan}
Let $S$ be a Stein space.  Fix $s\in S$. Let $f:X\to Y$ be a morphism interiorly of finite type between $\cO(S)$-schemes interiorly locally of finite type.
\begin{enumerate}[(i)]
\item If $f_{\cO_{S,s}}:X_{\cO_{S,s}}\to Y_{\cO_{S,s}}$ has dense image, there exists a Stein neighborhood $U\subset S$ of $s$ such that $(f_{\cO(U)})^{\an}:(X_{\cO(U)})^{\an}\to (Y_{\cO(U)})^{\an}$ has dense image.
\item If $f_{\cO_{S,s}}$ is a codimension $c$ closed embedding of regular schemes, there exists a Stein neighborhood $U\subset S$ of $s$ such that $(f_{\cO(U)})^{\an}$ is a codimension $c$ closed embedding of complex manifolds.
\end{enumerate}
\end{prop}

\begin{proof}
Assertion (i) is \cite[Satz 3.1 (6)]{Bingener} applied with $K=\{s\}$.

As for assertion (ii), it follows from \cite[(2.7)]{Bingener} (applied with $K=\{s\}$ and to the regularity property,  see \cite[(2.4) (2)]{Bingener}) that we may assume, after shrinking~$S$, that $X^{\an}$ and $Y^{\an}$ are manifolds.
After further shrinking $S$, one can also assume that $f^{\an}$ is a closed embedding (see \cite[Satz~3.1]{Bingener}).  The codimension assertion can then be ensured after shrinking~$S$ by applying \mbox{\cite[(2.7)]{Bingener}} with~$K=\{s\}$ to the property  \mbox{\cite[(2.4)~(4)]{Bingener}} of the ideal sheaf of $Y$ in~$X$.
\end{proof}

Lemmas \ref{openanal} and \ref{finiteanal} give examples of $\cO(S)$-schemes interiorly of finite type with interesting analytifications.

\begin{lem}
\label{openanal}
Let $Z$ be a closed analytic subspace of a Stein space $S$. Then there exists an open immersion $j:V\hookrightarrow \Spec(\cO(S))$ interiorly of finite type such that~$j^{\an}$  identifies with $S\setminus Z\hookrightarrow S$. 
If $Z$ is set-theoretically defined by the vanishing of finitely many elements of~$\cO(S)$, one may moreover choose $j$ to be quasi-compact.
\end{lem}

\begin{proof}
Set $V:=\Spec(\cO(S))\setminus \Spec(\cO(S)/\cI_Z(S))$, where $\cI_Z$ be the ideal sheaf of $Z$ in $S$.
Let $U\subset S$ be a relatively compact Stein open subset. By Lemma \ref{bcStein} below, 
$$V_U:=V\times_{\Spec(\cO(S))}\Spec(\cO(U))=\Spec(\cO(U))\setminus \Spec(\cO(U)/\cI_Z(U))$$
and $\cI_Z(U)$ is finitely generated over $\cO(U)$.  We deduce that $V_U$ is an $\cO(U)$-scheme of finite type.
The description of $j^{\an}$ follows from the concrete construction of the analytification that is recalled above (see also \cite[Proof of Satz 1.1]{Bingener}).

If $Z$ is set-theoretically defined by the vanishing of finitely many elements of~$\cO(S)$, we apply the above construction after replacing $Z$ with the complex subspace defined by the vanishing of these equations. This ensures that $\cI_Z(S)$ is a finitely generated ideal of $\cO(S)$ and hence that $j$ is quasi-compact.
\end{proof}

\begin{lem}
\label{finiteanal}
If $p:T\to S$ is a finite holomorphic map of Stein spaces,  then the induced morphism $f:\Spec(\cO(T))\to\Spec(\cO(S))$ is interiorly finite and $f^{\an}=p$.
\end{lem}

\begin{proof}
Let $U\subset S$ be a relatively compact Stein open subset.
By Lemma \ref{bcStein} below applied to $\cF:=p_*\cO_T$, the $\cO(U)$-module $\cO(T)\otimes_{\cO(S)}\cO(U)$ is finitely generated. It follows that $f$ is interiorly finite.  As $\cO(T)\otimes_{\cO(S)}\cO(U)\isoto \cO(p^{-1}(U))$ by Lemma~\ref{bcStein},  that $f^{\an}=p$ may be verified locally on $S$. We may thus assume that~$S$, hence also $T$,  is a finite-dimensional Stein space.  Under this assumption, it is shown in~\cite[(1.3)]{Bingener} that the analytifications  of $\Spec(\cO(T))$ viewed as an~$\cO(S)$\nobreakdash-scheme or as an~$\cO(T)$-scheme coincide. This exactly means that~$f^{\an}=p$.
\end{proof}

\begin{lem}
\label{bcStein}
Let $U$ be a relatively compact Stein open subset of a Stein space $S$. Let $\cF$ be a coherent sheaf on $S$. The $\cO(U)$-module $\cF(U)$ is generated by finitely many elements of $\cF(S)$, and $\cF(S)\otimes_{\cO(S)}\cO(U)\to\cF(U)$ is an isomorphism. 
\end{lem}

\begin{proof}
As $U$ is relatively compact,  Cartan's Theorem A implies that there exist $a_1,\dots,a_r\in\cF(S)$ generating $\cF$ on $U$.  The first assertion and the surjectivity of $\cF(S)\otimes_{\cO(S)}\cO(U)\to\cF(U)$ follow from the vanishing of $H^1(U, \Ker(\cO_U^{\oplus r}\xrightarrow{a_i}\cF|_U))$.

Now, let $\sum_{j=1}^s b_j\otimes c_j\in \cF(S)\otimes_{\cO(S)}\cO(U)$ be such that $\sum_{j=1}^s b_jc_j=0$ in $\cF(U)$. Consider the exact sequence
$0\to\cN\to \cO_S^{\oplus s}\xrightarrow{b_j}\cF,$ so $(c_j)\in\cN(U)$.  By the surjectivity result applied to $\cN$, there exist $(d_{j,1}),\dots,(d_{j,t})\in\cN(S)\subset\cO(S)^{\oplus s}$ and $e_{1},\dots,e_{t}\in \cO(U)$ with $\sum_{k=1}^t e_kd_{j,k}=c_j$ for $1\leq j\leq s$. Then
\phantom{\qedhere}
$$\sum_{j=1}^s b_j\otimes c_j=\sum_{j=1}^s\sum_{k=1}^t b_j\otimes e_kd_{j,k}=\sum_{k=1}^t\Big(\sum_{j=1}^sb_jd_{j,k}\Big)\otimes e_k=0.\eqno\qed$$
\end{proof}

\subsection{\'Etale sheaves on algebraic varieties over Stein algebras}
\label{paretale}

Let $S$ be a Stein space, and let $X$ be an $\cO(S)$\nobreakdash-scheme interiorly locally of finite type. There is a commutative diagram of morphisms of sites 
\begin{equation}
\label{carredesites}
\begin{aligned}
\xymatrix@C=1.5em@R=3ex{
(X^{\an})_{\cl}\ar^{\delta}[r]\ar^{\varepsilon}[d]&X^{\an}\ar^{i_X}[d] \\
X_{\et}\ar[r]& X\rlap{,}
}
\end{aligned}
\end{equation}
where $X_{\et}$ is the small \'etale site of the scheme $X$, and $X^{\an}_{\cl}$ is the site of local isomorphisms of the topological space $X^{\an}$ (see \cite[XI, \S 4.0]{SGA43}). Indeed, if $f:Y\to X$ is \'etale, then $f^{\an}:Y^{\an}\to X^{\an}$ is a local isomorphism by Proposition~\ref{permanencean}. It follows from \cite[III, Th\'eor\`eme 4.1]{SGA41} that $\delta_*$ induces an equivalence of topoi. 

If $\L$ is an abelian sheaf on $X_{\et}$, we define $\L^{\an}:=\delta_*\varepsilon^*\L$.  For $x\in X^{\an}$, one still denotes by $x\in X$ the associated closed point with complex residue field, and one has $\L^{\an}_x=\L_x$. Finally, we say that~$\L$ is \textit{interiorly constructible} if there exists an open covering $(U_i)_{i\in I}$ of~$S$ such that the \'etale sheaves $\L|_{X_{\cO(U_i)}}$ are constructible.

\begin{lem}
\label{cdbK}
Let $K$ be an excellent Stein compact subset of a Stein space $S$.
Let $f:X\to Y$ be a morphism interiorly of finite type of $\cO(S)$\nobreakdash-schemes interiorly locally of finite type.  Fix a torsion abelian sheaf $\L$ on~$X_{\et}$. Then the base change morphisms 
$(\RR^kf_*\L)|_{Y_{\cO(K)}}\to \RR^k(f_{\cO(K)})_*(\L|_{X_{\cO(K)}})$ are isomorphisms for $k\geq 0$.
\end{lem}

\begin{proof}
Fix $k\geq 0$.
As $\cO(S)\to\cO(S)_K$ is a localization, the description of the fibers of higher push-forwards in \'etale cohomology given in \cite[Theorem \href{https://stacks.math.columbia.edu/tag/03Q9}{03Q9}]{SP} implies that the base change morphism 
\begin{equation}
\label{eq1lem}
\RR^k f_*(\L)|_{Y_{\cO(S)_K}}\to\RR^k(f_{\cO(S)_K})_*(\L|_{X_{\cO(S)_K}})
\end{equation}
is an isomorphism. By N\'eron--Popescu desingularization (see \cite[Theorem~1.1]{Swan})
and Lemma \ref{regular}, the ring $\cO(K)$ is a filtered colimit of smooth $\cO(S)_K$-algebras.
By smooth base change (see \cite[Lemma \href{https://stacks.math.columbia.edu/tag/0F09}{0F09}]{SP}), the base change morphism
\begin{equation}
\label{eq2lem}
\big(\RR^k (f_{\cO(S)_K})_*(\L|_{X_{\cO(S)_K}})\big)|_{Y_{\cO(K)}}\to\RR^k (f_{\cO(K)})_*(\L|_{X_{\cO(K)}})
\end{equation}
is also an isomorphism. The lemma follows from (\ref{eq1lem}) and (\ref{eq2lem}).
\end{proof}

\begin{lem}
\label{constr}
Let $S$ be a Stein space. 
Let $f:X\to Y$ be a morphism interiorly of finite type of $\cO(S)$\nobreakdash-schemes interiorly locally of finite type.  Let $\L$ be an interiorly constructible sheaf on~$X_{\et}$. Then $\RR^kf_*\L$ is interiorly constructible for $k\geq 0$.
\end{lem}

\begin{proof}
Working locally on $Y$, we may suppose that $X$ and $Y$ are interiorly of finite type over $\cO(S)$.
Fix $s\in S$.  Let $K$ be an excellent Stein compact neighborhood of~$s$ in $S$ chosen small enough so that 
$\L|_{X_{\cO(K)}}$ is constructible (see Lemma~\ref{semianal}~(ii)).  As $f_{\cO(K)}:X_{\cO(K)}\to Y_{\cO(K)}$ is a morphism of  $\cO(K)$-schemes of finite type (see Lemma~\ref{int}~(i)), which are excellent noetherian schemes by Lemma \ref{excellent} (iii), the sheaf $\RR^k (f_{\cO(K)})_*(\L|_{X_{\cO(K)}})$ is constructible by \cite[XIX, Th\'eor\`eme 5.1]{SGA43} and \cite[Theorem 1.1]{Temkin}.  It follows from Lemma \ref{cdbK} that $(\RR^k f_*\L)|_{X_{\cO(U)}}$ is constructible for any open neighborhood $U$ of $s$ in $K$.
\end{proof}

\subsection{A relative Artin comparison theorem over Stein spaces}
\label{parrelArtin}

Here is the statement of our relative comparison theorem. In the constructible case, its proof is entirely parallel to that of \cite[XVI, Th\'eor\`eme~4.1]{SGA43}.

\begin{thm}
\label{relArtin}
Let $S$ be a Stein space.
Let $f:X\to Y$ be a morphism interiorly of finite type of $\cO(S)$-schemes interiorly locally of finite type. 
Let $\L$ be a torsion abelian sheaf on $X_{\et}$. If either $f$ is interiorly proper or $\L$ is interiorly constructible, the base change morphisms $\RR^k f_*(\L)^{\an}\to\RR^k f^{\an}_*(\L^{\an})$ are isomorphisms for $k\geq 0$.
\end{thm}

\begin{proof}Fix $y\in Y^{\an}$ and let $s$ be the image of $y$ in $S$. To show that the morphism 
$\RR^k f_*(\L)_y^{\an}\to\RR^k f^{\an}_*(\L^{\an})_y$ is an isomorphism for $k\geq 0$,  we proceed in several~steps. 

\begin{Step}
\label{step1}
We first deal with the case where $f$ is interiorly proper.
\end{Step}

By Lemma \ref{int} (i), the base change $f_{\cO_{S,s}}:X_{\cO_{S,s}}\to Y_{\cO_{S,s}}$ of~$f$
is proper. By Lemma \ref{cdbK}, proper base change \cite[Theorem \href{https://stacks.math.columbia.edu/tag/095T}{095T}]{SP}, and Artin's comparison theorem \cite[XVI, Th\'eor\`eme 4.1]{SGA43}, the base change morphisms
 \begin{equation}
\label{eq1}
(\RR^kf_*\L)_y\to\big(\RR^k (f_{\cO_{S,s}})_*(\L|_{X_{\cO_{S,s}}})\big)_y\to H^k_{\et}(X_y,\L|_{X_y})\to H^k(X_y^{\an},\L^{\an}|_{X^{\an}_y})
\end{equation}
are isomorphisms. As $f^{\an}$ is proper by Proposition \ref{permanencean},  proper base change in topology (see \cite[III, Theorem 6.2]{Iversen}) implies that the base change morphism
\begin{equation}
\label{eq2}
\RR^k f^{\an}_{*}(\L^{\an})_y\to H^k(X_y^{\an},\L^{\an}|_{X^{\an}_y})
\end{equation}
is also an isomorphism. Combining (\ref{eq1}) and (\ref{eq2}) shows that the base change morphism $\RR^k f_*(\L)_y^{\an}\to\RR^k f^{\an}_*(\L^{\an})_y$ is an isomorphism, as wanted.

\begin{Step}
\label{step2}
Setup of the proof when $\L$ is interiorly constructible.
\end{Step}

By Lemma \ref{cdbK}, if $U\subset S$ is a Stein open neighborhood of $s$, the base change morphism $(\RR^kf_*\L)_y\to\big(\RR^k(f_{\cO(U)})_*(\L|_{X_{\cO(U)}})\big)_y$ is an isomorphism.  Therefore, to check that $\RR^k f_*(\L)_y^{\an}\to\RR^k f^{\an}_*(\L^{\an})_y$ is an isomorphism, we may at any time replace~$S$ with a Stein open neighborhood of~$s$ in~$S$. As one can moreover work locally on~$Y$, one may assume that $f$ is a morphism of finite presentation between $\cO(S)$\nobreakdash-schemes of finite presentation, that $\L$ is constructible, that $Y$ is separated, and that $X^{\an}$ is finite-dimensional.
We argue by induction on the dimension of~$X^{\an}$.

\begin{Step}
\label{step3}
We reduce to the case where $f$ is an open immersion 
of reduced separated $\cO(S)$-schemes of finite presentation such that $f_{\cO_{S,s}}$ has dense image, and $\L=\Z/m$.
\end{Step}

By \cite[Lemmas \href{https://stacks.math.columbia.edu/tag/09Z7}{09Z7}, \href{https://stacks.math.columbia.edu/tag/095R}{095R} and \href{https://stacks.math.columbia.edu/tag/03RX}{03RX}]{SP}, the sheaf $\L$ admits a resolution
\begin{equation}
\label{resolution1}
0\to\L\to\L_0\to\L_1\to\dots
\end{equation}
by constructible sheaves $\L_p$ that are finite products of sheaves of the form $\pi_*\M$ with $\pi:Z\to X$ finite of finite presentation and~$\M$ constant constructible on $Z_{\et}$.  Making use of the spectral sequences
\begin{alignat}{4}
&E_1^{p,q}=\RR^qf_*(\L_p)^{\an} &\implies& \RR^{p+q}f_*(\L)^{\an}\textrm{\hspace{1em} and }\nonumber\\
&E_1^{p,q}=\RR^qf^{\an}_*(\L_p^{\an})&\implies& \RR^{p+q}f^{\an}_*(\L^{\an})\nonumber
\end{alignat}
associated with the resolution (\ref{resolution1}), we reduce to the case where $\L$ is of the form~$\pi_*\M$. 
As the higher direct images of $\pi$ and $\pi^{\an}$ vanish by \cite[Proposition~\href{https://stacks.math.columbia.edu/tag/03QP}{03QP}]{SP} and \cite[III, Theorem 6.2]{Iversen}, and as $\pi_*(\M)^{\an}\isoto\pi^{\an}_*(\M^{\an})$ by Step \ref{step1}, we may further assume that $\L$ is constant (after replacing $X$ with~$Z$ and~$\L$ with~$\M$), and then that $\L\simeq \Z/m$.

Let $\cU=(U_i)_{i\in I}$ be a finite affine covering of $X$. Computing cohomology using \v{C}ech spectral sequences associated with $\cU$, we may replace $f$ with $f|_{U_J}$, where $U_J:=\cap_{i\in J}\,U_i$ for $J\subset I$. Since $X$ was quasi-separated (as is any $\cO(S)$-scheme of finite presentation), this trick reduces us to the case where $X$ is separated.

As the \'etale and classical topologies are insensitive to nilpotents (see \cite[Proposition~\href{https://stacks.math.columbia.edu/tag/03SI}{03SI}]{SP}), we may assume that $X$ and $Y$ are reduced. 

Write~$f_{\cO_{S,s}}$ as the composition $X_{\cO_{S,s}}\xrightarrow{g_s}\oX_s\xrightarrow{h_s} Y_{\cO_{S,s}}$ of an open immersion~$g_s$ with dense image and of a proper morphism~$h_s$ using Nagata's compactification theorem \cite[Theorem \href{https://stacks.math.columbia.edu/tag/0F41}{0F41}]{SP}.  Use a limit argument based on
\cite[Lemmas \href{https://stacks.math.columbia.edu/tag/01ZM}{01ZM}, \href{https://stacks.math.columbia.edu/tag/086Q}{086Q}, \href{https://stacks.math.columbia.edu/tag/0EUU}{0EUU} and \href{https://stacks.math.columbia.edu/tag/081F}{081F}]{SP} to extend, after shrinking~$S$, these morphisms of schemes to a diagram $X\xrightarrow{g}\oX\xrightarrow{h} Y$ of reduced separated $\cO(S)$\nobreakdash-schemes of finite presentation, where~$g$ is an open immersion and~$h$ is proper.
As Theorem \ref{relArtin} holds for $h$ (and any torsion \'etale sheaf) by Step~\ref{step1}, the composite functors spectral sequence computing~$\RR^k(h\circ g)_*\Z/m$ reduces us to proving it for $g$.

\begin{Step}
\label{step4}
We reduce to the case where moreover $Y_{\cO_{S,s}}$ is regular.
\end{Step}

Let $\mu_s:\wY_{s}\to Y_{\cO_{S,s}}$ be a resolution of singularities (use \cite[Theorem 1.1]{Temkin} and Lemma \ref{excellent} (iii)). Set $\wX_{s}:=X\times_{Y_{\cO_{S,s}}}\wY_{s}$, let $j_s:U_{s}\hookrightarrow X_{s}$ be a dense open subset over which $\mu_s$ is an isomorphism, and let $\tj_s:U_s\to\wX_s$ be the lift of~$j_s$. By a limit argument based on \cite[Lemmas \href{https://stacks.math.columbia.edu/tag/01ZM}{01ZM}, \href{https://stacks.math.columbia.edu/tag/0EUU}{0EUU} and \href{https://stacks.math.columbia.edu/tag/081F}{081F}]{SP}, the above schemes and morphisms arise by base change, after shrinking $S$, from a cartesian diagram
\begin{equation*}
\begin{aligned}
\xymatrix@C=1.5em@R=3ex{
U\ar@{=}[d]\ar^{\tj}[r]&\wX\ar^{\mu_X}[d]\ar^{\tf}[r]&\wY\ar^{\mu}[d] \\
U\ar^{j}[r]&X\ar^{f}[r]& Y\rlap{,}
}
\end{aligned}
\end{equation*}
whose horizontal arrows are quasi-compact open immersions and whose vertical arrows are proper of finite presentation. Let $i:Z\to X$ and $\ti:\wZ\to\wX$ be the inclusions of the complements of $U$ in~$X$ and $\wX$ respectively.
As $j_s$ and $\tj_s$ have dense image,  both $\dim(Z^{\an})$ and $\dim((\wZ)^{\an})$ are $<\dim(X^{\an})$ after possibly shrinking $S$ by Proposition \ref{extan} (i).  
The induction hypothesis thus shows that Theorem \ref{relArtin} holds for $(f\circ i,\Z/m)$ and $(\tf\circ \ti,\Z/m)$, hence for
$(f,i_*\Z/m)$ and $(\tf,\ti_*\Z/m)$.

The five lemma applied to long exact sequences induced by the short exact sequence $0\to j_!\,\Z/m\to\Z/m\to i_*\Z/m\to 0$ now shows that, to prove Theorem \ref{relArtin} for $(f,\Z/m)$, it suffices to prove it for $(f,j_!\,\Z/m)$. 

One computes using proper base change in \'etale cohomology \cite[Theorem~\href{https://stacks.math.columbia.edu/tag/095T}{095T}]{SP} that $(\mu_{X})_*(\tj_!\,\Z/m)=j_!\,\Z/m$ and $\RR^q(\mu_{X})_*(\tj_!\,\Z/m)=0$ if $q>0$, and using proper base change in topology \cite[III, Theorem 6.2]{Iversen} that $(\mu_{X}^{\an})_*(\tj^{\an}_!\,\Z/m)=j^{\an}_!\,\Z/m$ and $\RR^q(\mu_{X}^{\an})_*(\tj^{\an}_!\,\Z/m)=0$ if $q>0$.
Using the composite functors spectral sequences computing $\RR^k(f\circ\mu_X)_*$ and $\RR^k((f\circ\mu_X)^{\an})_*$, we deduce that it suffices to prove Theorem \ref{relArtin} for $(\tf,\tj_!\,\Z/m)$.  By another application of the five lemma based on the short exact sequence $0\to \tj_!\,\Z/m\to\Z/m\to \ti_*\Z/m\to 0$, it remains to prove Theorem \ref{relArtin} for $(\tf,\Z/m)$.

\begin{Step}
\label{step5}
Under the additional hypotheses to which Steps \ref{step3} and \ref{step4} have reduced us to, we prove the theorem in the equivalent form (in view of Lemma \ref{cdbK}) that the natural morphisms $\RR^k (f_{\cO_{S,s}})_*(\Z/m)_y\to\RR^k f^{\an}_*(\Z/m)_y$ are isomorphisms for $k\geq 0$.
\end{Step}

The noetherian scheme $Y_{\cO_{S,s}}\setminus f_{\cO_{S,s}}(X_{\cO_{S,s}})$ is excellent by Lemma~\ref{excellent}~(iii). Stratify it by its regular locus, the regular locus of its singular locus, etc. Using
\cite[Lemmas \href{https://stacks.math.columbia.edu/tag/01ZM}{01ZM},  \href{https://stacks.math.columbia.edu/tag/01ZP}{01ZP} and  \href{https://stacks.math.columbia.edu/tag/0EUU}{0EUU}]{SP},  lift this stratification to a stratification of~${Y\setminus f(X)}$, after maybe shrinking $S$.  One can thus write $f=f_r\circ\dots\circ f_1$ where ${f_i:X_{i-1}\to X_i}$ is a quasi-compact open immersion (with $X_0:=X$ and $X_r:=Y$), 
with complementary closed immersion $g_i:Z_i=X_i\setminus f_i(X_{i-1})\to X_i$,
so that~$f_{i,\cO_{S,s}}$ has dense image and $(Z_i)_{\cO_{S,s}}$ is regular. 

After shrinking~$Y$, we may assume that $g_{i,\cO_{S,s}}$ is a regular immersion of pure codimension~$c_i\geq 1$ (for $1\leq i\leq r$), and that $y\notin (f_{r-1}(X_{r-1}))^{\an}$ (after possibly decreasing $r$). Using Proposition \ref{extan} (ii), one can ensure after shrinking $S$ that~$g_i^{\an}$ is a closed embedding of complex manifolds of pure codimension $c_i$.
It follows from these facts that $\Z/m\isoto (f_i^{\an})_*\Z/m$ and, in view of \cite[XIX, Th\'eor\`eme~2.1]{SGA43},
that $\Z/m\isoto (f_{i,\cO_{S,s}})_*\Z/m$.

Assuming for now that Theorem \ref{relArtin} holds for all the $(f_i,\Z/m)$, we prove by decreasing induction on $i$ that it holds for $(f_r\circ\dots\circ f_i,\Z/m)$.  Taking $i=1$ concludes the proof of the theorem.  To perform the induction step,  we make use of the composite functors spectral sequences computing $\RR^k((f_{r,\cO_{S,s}}\circ\dots\circ f_{i+1,\cO_{S,s}})\circ f_{i,\cO_{S,s}})_*\Z/m$ and $\RR^k((f_{r}\circ\dots\circ f_{i+1})^{\an}\circ f_i^{\an})_*\Z/m$ respectively, noting that Theorem \ref{relArtin} holds for $(f_{r}\circ\dots\circ f_{i+1},\RR^q(f_{i})_*\Z/m)$ when~$q>0$ by the induction hypothesis (because the sheaf $\RR^q(f_{i})_*\Z/m$ is interiorly constructible by Lemma \ref{constr} and because it is supported on $Z_i$ which satisfies $\dim(Z_i^{\an})<\dim(X_i^{\an})=\dim(X^{\an})$).

It remains to prove the theorem for $(f_i,\Z/m)$.  It follows from the Thom isomorphism theorem that $(\RR^k (f_i^{\an})_*\Z/m)_y=\Z/m$ if $k\in\{0,2c-1\}$ and that this group vanishes otherwise. 
 It follows from the purity isomorphism (the analogue of the Thom isomorphism theorem in \'etale cohomology, for which see \cite[XIX, Th\'eor\`emes~3.2 and 3.4]{SGA43}) that $(\RR^k(f_{i,\cO_{S,s}})_*\Z/m)_y=\Z/m$ if $k\in\{0,2c-1\}$ and that this group vanishes otherwise. 
Inspecting the construction of these isomorphisms shows that they are compatible,  and completes the proof of the theorem.
\end{proof}

\begin{rem}
Theorem \ref{relArtin} fails if one drops both the properness and the constructibility hypotheses, already if $S$ is a point, if $f:\A^1_{\C}\to \Spec(\C)$ is the structural morphism, if $k=0$, and if $\L$ is a direct sum of skyscraper sheaves at the integers. 
\end{rem}

\section{Killing ramification on finite coverings}
\label{parram}

The goal of this section is result is Proposition \ref{thkillram}. It will be used in conjunction with Proposition \ref{killunram} in the proof of our Theorem \ref{Artincompcx}.

\subsection{Killing ramification locally}
\label{parkill1}

We first apply Theorem \ref{relArtin} to kill the ramification of cohomology classes in the local analytic setting.

\begin{prop}
\label{lemkillloc}
Let $U$ be an open subset of a connected normal Stein space~$S$. Choose $s\in U$.
Let $Z\subset S$ be a closed analytic subset. Fix $\alpha\in H^k(U\setminus (Z\cap U),\Z/m)$ with $k,m\geq 1$. There exist a finite surjective holomorphic map $q:R\to S$ of connected normal Stein spaces, a point $r\in R$ with $q(r)=s$, and an open  neighborhood~$\Theta$ of $r$ in $q^{-1}(U)$ with $\alpha|_{\Theta\setminus (q^{-1}(Z)\cap\Theta)}=0$ in $H^k(\Theta\setminus (q^{-1}(Z)\cap\Theta),\Z/m)$.
\end{prop}

\begin{proof}
We may assume that $Z$ is nowhere dense in $S$ (otherwise $Z=S$ and the lemma is trivial).
Let $\cI_Z\subset\cO_S$ be the ideal sheaf of $Z$. Use Cartan's Theorem~A to choose finitely many elements $(f_i)_{i\in I}$ of $\cI_Z(S)$ that generate~$\cI_{Z,s}$. After replacing $Z$ with the vanishing locus of the $f_i$, we may apply Lemma \ref{openanal} to find a quasi-compact open immersion ${j:V\hookrightarrow \Spec(\cO(S))}$ such that $j^{\an}$ identifies with $S\setminus Z\hookrightarrow S$.
Theorem \ref{relArtin} then shows that the base change morphism
\begin{equation}
\label{bciso}
(\RR^k j_*\Z/m)^{\an}\to\RR^k j^{\an}_*\Z/m
\end{equation}
is an isomorphism.
Let $\alpha_s$ be the image of $\alpha$ in $(\RR^k j^{\an}_*\Z/m)_s$ and consider its inverse image $\beta_s\in( \RR^k j_*\Z/m)_{s}$ by the isomorphism (\ref{bciso}). 

Let $\beta\in H^k_{\et}(X_V,\Z/m)$ be a representative of $\beta_s$, where $f:X\to\Spec(\cO(S))$ is an integral affine \'etale neighborhood of $s\in\Spec(\cO(S))$ and $X_V:=X\times_{\Spec(\cO(S))}V$.  
The scheme $(X_V)_s:=X_V\times_{\Spec(\cO(S))}\Spec(\cO(S)_{s})$ is \'etale of finite type over~$\cO(S)_{s}$ (use Lemma~\ref{int}~(ii)), hence an excellent noetherian scheme by Lemma~\ref{excellent}~(ii).  
By \cite[Theorem 1.1]{Bhatt}, 
there exists a finite surjective morphism $\pi_s:Y_s\to (X_V)_s$ such that $\beta|_{Y_s}=0$ in $H^k_{\et}(Y_s,\Z/m)$ (note that Bhatt's proof simplifies in our situation, as the group scheme $\Z/m$ is \'etale, see \cite[Remark 3.3]{Bhatt}).

Our next goal is to globalize the morphism $\pi_s$.
By a limit argument based on \cite[Lemmas \href{https://stacks.math.columbia.edu/tag/01ZM}{01ZM},  \href{https://stacks.math.columbia.edu/tag/01ZO}{01ZO} and \href{https://stacks.math.columbia.edu/tag/07RR}{07RR}]{SP}, there exist an affine open neighborhood $W$ of $s$ in $\Spec(\cO(S))$ and a finite surjective morphism of finite presentation ${\pi:Y\to (X_V)\times_{\Spec(\cO(S))}W}$ such that the base change of~$\pi$ by $\Spec(\cO(S)_{s})\to W$ identifies with $\pi_s$. By \cite[Theorem \href{https://stacks.math.columbia.edu/tag/09YQ}{09YQ}]{SP}, one may assume after shrinking~$W$ that~$\beta|_Y=0$ in $H^k_{\et}(Y,\Z/m)$. After replacing~$X$ with $X\times_{\Spec(\cO(S))}W$, we may finally view $\pi$ as a finite surjective morphism of finite presentation $\pi:Y\to X_V$ such that~$\pi^*\beta=0$. As $X_V$ is integral, we may assume that so is $Y$ (after replacing it with an irreducible component that dominates $X_V$).

One may factor ${f:X\to \Spec(\cO(S))}$ as the composition $X\xrightarrow{a}\oX\xrightarrow{b}\Spec(\cO(S))$ of a quasi-compact open immersion $a$ and of a finite morphism $b$, by Zariski's Main Theorem 
\cite[Lemma \href{https://stacks.math.columbia.edu/tag/05K0}{05K0}]{SP}.
A second application of Zariski's Main Theorem allows us to factor the natural morphism $Y\to \oX$ as the composition $Y\xrightarrow{c}\oY\xrightarrow{d}\oX$ of a quasi-compact open immersion $c$ and a finite morphism $d$. After replacing $\oX$ with the closure of $X$ in $\oX$ and $\oY$ with the closure of $Y$ in $\oY$, one may assume that~$\oX$ and $\oY$ are integral. Here is a diagram summarizing the situation:
\begin{equation*}
\begin{aligned}
\xymatrix@C=1.5em@R=3ex{
Y\ar_{\pi}[d]\ar^{c}[rr]&&\oY\ar^{d}[d] \\
X_{V}\ar[r]\ar[d]&X\ar^{a}[r]\ar_(.35){\hspace{1em}f}[rd]& \oX\ar^{b}[d]\\
V\ar_{j\hspace{2em}}[rr]&&\Spec(\cO(S))\rlap{.}
}
\end{aligned}
\end{equation*}
As the natural morphism $Y\to \oY\times_{\oX}X_V$ is a proper open immersion with nonempty source and integral target, it is an isomorphism. 
Choose $x\in X^{\an}$ with $f^{\an}(x)=s$.  
As $f$ is \'etale, the map $f^{\an}$ is a local biholomorphism at $x$ (see Proposition \ref{permanencean}) so that $X^{\an}$ is normal at $x$, hence locally irreducible at $x$. 
As $(X_V)^{\an}$ is dense in~$X^{\an}$ (because $Z$ is nowhere dense in $S$), and as $\pi^{\an}$ is finite surjective and $d^{\an}$ is finite by Proposition \ref{permanencean}, the restriction $(\oY\times_{\oX}X)^{\an}\to X^{\an}$ of $d^{\an}$ is finite and surjective. One may thus choose $y\in(\oY\times_{\oX}X)^{\an}\subset (\oY)^{\an}$ such that $d^{\an}(y)=x$ and moreover, by local irreducibility of $X^{\an}$ at $x$, such that the image in $(\oX)^{\an}$ of some irreducible component $C$ of $(\oY)^{\an}$ through $y$ contains a neighborhood of $x$.

Consider the natural commutative diagram
\begin{equation}
\label{diagcdb}
\begin{aligned}
\xymatrix@C=1.5em@R=3ex{
(\RR^k j_*\Z/m)_s\ar[r]\ar[d]&(\RR^k j^{\an}_*\Z/m)_s\ar[d] \\
(\RR^k c_*\Z/m)_y\ar[r]&(\RR^k c^{\an}_*\Z/m)_y \rlap{.}
}
\end{aligned}
\end{equation}
As $\pi^*\beta=0$, the image of $\beta_s$ in $(\RR^kc_{*}\Z/m)_y$ vanishes. By commutativity of (\ref{diagcdb}), the image of $\alpha_s$ in $(\RR^kc^{\an}_{*}\Z/m)_y$ also vanishes. This exactly means that there exists an open neighborhood $\Theta'$ of $y$ in $(b^{\an}\circ d^{\an})^{-1}(U)\subset(\oY)^{\an}$ such that $\alpha|_{Y^{\an}\cap\Theta'}=0$. If $\Theta'$ is small enough to be included in $(d^{\an})^{-1}(X^{\an})$, one has the equality 
$$Y^{\an}\cap\Theta'=\Theta'\setminus((b^{\an}\circ d^{\an})^{-1}(Z)\cap\Theta').$$

To conclude, it suffices to take $R$ to be the normalization of $C\subset\oY^{\an}$, to let~$r$ be any preimage of $y$ in $R$, and to choose $\Theta$ to be the inverse image of $\Theta'$ in $R$. 
\end{proof}

\begin{prop}
\label{killloc}
Let $S$ be a Stein space,  let $U\subset S$ be open, and choose $s\in U$. Let $Z\subset S$ be a closed analytic subset. Fix $\alpha\in H^k(U\setminus (Z\cap U),\Z/m)$ with $k,m\geq 1$. Then there exist a finite surjective holomorphic map $p:T\to S$ and an open neighborhood $\Omega$ of $s$ in $U$ such that $\alpha|_{p^{-1}(\Omega\setminus(Z\cap\Omega))}=0$
 in $H^k(p^{-1}(\Omega\setminus(Z\cap\Omega)),\Z/m)$.
\end{prop}

\begin{proof} 
\setcounter{Step}{0}
Let $\nu:\wS\to S$ be the normalization map and let $(\ts_i)_{i\in I}$ be the preimages of $s$ by~$\nu$. Assume that the maps $p_i:T_i\to\wS$ prove the proposition for $(\wS, \nu^{-1}(U), \ts_i, \nu^{-1}(Z), \nu^*\alpha)$ using disjoint open neighborhoods $\Omega_i$ of $\ts_i$ in $\wS$. Then the fiber product $p:T\to S$ of the maps $\nu\circ p_i$ solves our original problem for any neighborhood $\Omega$ of $s$ in $S$ such that $\nu^{-1}(\Omega)\subset\cup_{i\in I}\,\Omega_i$. We may thus assume that~$S$ is normal. After replacing $S$ with its connected component containing $s$, we may also suppose that it is connected.

Choose $q:R\to S$, a point $r\in R$ and $\Theta\subset q^{-1}(U)$ as in Proposition \ref{lemkillloc}.
Let $p':T\to R$ and $\Gamma$ be given by Lemma~\ref{Galoisclosure}, and set $p:=q\circ p'$.
Let $(t_j)_{j\in J}$ be the preimages of~$s$ by $p$. As $p'$ is surjective and $\Gamma$ acts transitively on the fibers of~$p$, one can find, for each $j\in J$, an element $\gamma_j\in \Gamma$ such that $p'\circ\gamma_j(t_j)=r$. Let $\Omega_j$ be a neighborhood of $t_j$ in $T$ such that $p'\circ\gamma_j(\Omega_j)\subset \Theta$. After shrinking the $\Omega_j$, we may assume that they are disjoint. Choosing $\Omega$ to be a neighborhood of $s$ in $U$ such that $p^{-1}(\Omega)\subset\cup_{j\in J}\,\Omega_j$ concludes the proof.
\end{proof}

\subsection{Killing ramification globally}
\label{parkill2}

We finally globalize Proposition~\ref{killloc} thanks to the \v{C}ech-to-derived spectral sequence.

\begin{prop}
\label{thkillram}
Let $S$ be a Stein space,  let $K\subset S$ be a compact subset and let~$U$ be an open neighborhood of $K$ in $S$. Let $Z\subset S$ be a closed analytic subset, set $V:=S\setminus Z$, and fix $\alpha\in H^k(V\cap U,\Z/m)$ for some $k,m\geq 1$.  Then, after maybe shrinking $U$,  there exist a finite surjective holomorphic map $p:T\to S$ and a class ${\beta\in H^k(p^{-1}(U),\Z/m)}$ with $\alpha|_{p^{-1}(V\cap U)}=\beta|_{p^{-1}(V\cap U)}$  in $H^k(p^{-1}(V\cap U),\Z/m)$.
\end{prop}

\begin{proof}
Let $\cH^s$ be the presheaf on $S$ defined by $\cH^s(\Omega):=H^s(\Omega,\Z/m)$ for $\Omega\subset S$ open.  Note that the presheaf $\cH^0$ is the constant sheaf $\Z/m$. Let $\cU=(U_i)_{i\in I}$ be a finite open covering of $U$ and consider the \v{C}ech-to-derived spectral sequence
$$E_2^{r,s}=\check{H}^r(V\cap\cU,\cH^s)\implies H^{r+s}(V\cap U,\Z/m)$$
 associated with the open covering $V\cap\cU=(V\cap U_i)_{i\in I}$ of $V\cap U$ (\cite[Lemma~\href{https://stacks.math.columbia.edu/tag/03AZ}{03AZ}]{SP}, see \cite[Definition \href{https://stacks.math.columbia.edu/tag/03AM}{03AM}]{SP} for the definition of \v{C}ech cohomology). Let $F^{\bullet}$ be the filtration on its abutment. Let $l\geq 0$ be maximal with $\alpha\in F^{l}H^k(V\cap U,\Z/m)$.

 Assume first that $l<k$.  If $J=(i_0,\dots,i_{l})\in I^{l+1}$, write $U_J:=\cap_{t=0}^{l}U_{i_t}$. Then~$\alpha$ is induced by a class in $\check{H}^{l}(V\cap\cU,\cH^{k-l})$,  represented by a cocycle $(\gamma_J)_{J\in I^{l+1}}$, where 
$$\gamma_J\in \cH^{k-l}(V\cap U_J)=H^{k-l}(V\cap U_J,\Z/m).$$
 For each $s\in K$ and $J\in I^{l+1}$,  one can apply Proposition \ref{killloc} to find an open neighborhood $\Omega_{s,J}$ of $s$ in $U_J$ and a finite surjective holomorphic map ${p_{s,J}:T_{s,J}\to S}$ with $\gamma_J|_{p_{s,J}^{-1}(V\cap\Omega_{s,J})}=0$. Intersecting the $\Omega_{s,J}$ and taking the fiber product of the $p_{s,J}:T_{s,J}\to S$ over $S$ for varying $J$ yields an open neighborhood~$\Omega_s$ of~$s$ in~$U$ and a finite surjective holomorphic map $p_s:T_s\to S$.  Extract from $(\Omega_s)_{s\in K}$ a finite covering of $K$ (which we view, after shrinking $U$, as a finite covering~$\cU'$ of~$U$ which refines $\cU$) and let $p:T\to S$ be the fiber product over~$S$ of the corresponding maps~$p_{s}:T_{s}\to S$. Our choices ensure that the image  in $\check{H}^{l}(p^{-1}(V)\cap p^{-1}(\cU'),\cH^{k-l})$ of $[\gamma_J]\in \check{H}^{l}(V\cap\cU,\cH^{k-l})$ is represented by the zero cocycle and hence vanishes.  As a consequence, after replacing $S$, $K$, $V$ and $\cU$ with~$T$, ~$p^{-1}(K)$,~$p^{-1}(V)$ and~$p^{-1}(\cU')$, we have managed to increase the value of $l$.

Repeating this procedure finitely many times, we may assume that $l=k$, hence that 
$\alpha$ is induced by a class $\talpha\in\check{H}^k(V\cap\cU,\cH^0)=\check{H}^k(V\cap\cU,\Z/m)$. 
We may suppose that $S$ is normal and connected (replace it with a connected component of its normalization) and that $Z$ is nowhere dense in $S$ (otherwise $Z=S$ and $\alpha=0$). Then, if $j:V\hookrightarrow S$ is the inclusion morphism, the adjunction morphism $\Z/m\to j_*j^*\Z/m$ is an isomorphism (use \cite[9, \S 1.2, Theorem vii) $\Leftrightarrow$ viii)]{GRCoherent}).
As a consequence, the restriction map $\check{H}^k(\cU,\Z/m)\to\check{H}^k(V\cap\cU,\Z/m)$ is an isomorphism, and we let $\tbeta\in\check{H}^k(\cU,\Z/m)$ be the inverse image of $\talpha$. 
The image $\beta\in H^k(U,\Z/m)$ of $\tbeta$ by the natural morphism $\check{H}^k(\cU,\Z/m)\to H^k(U,\Z/m)$ then has the required property that $\beta|_{V\cap U}=\alpha$ in $H^k(V\cap U,\Z/m)$.
\end{proof}

\section{An absolute comparison theorem}
\label{parcomp}

After fixing in \S\ref{parqfh} our conventions concerning the Grothendieck topologies that we will use, we prove our main comparison theorem (Theorem \ref{Artincompcx}) in \S\ref{parArtincompact}, we extend it to the $\Gal(\C/\R)$\nobreakdash-equivariant setting in \S\ref{Geqpar} and \S\ref{Geqpar2}, and we give applications to cohomological dimension bounds in \S\ref{cdpar}.

\subsection{The qfh topology}
\label{parqfh}

Our proof of Theorem~\ref{Artincompcx} requires the use of a Grothen\-dieck topology for schemes that is finer than the \'etale topology.
The definition that we use is the one given in \cite[\S 6.6]{Chough}. 
To be precise, we say that a family $(f_i:Y_i\to X)_{i\in I}$ of morphisms of schemes that are locally quasi-finite and locally of finite presentation is a \textit{qfh covering}
of $X$ if for any affine open subset $U\subset X$, the induced family $(Y_i\times_X U\to U)_{i\in I}$ can be refined by a family $(g_j:Z_j\to U)_{j\in J}$, where~$g_j=g|_{Z_j}$ for some proper surjective morphism $g:Z\to X$ and some affine covering $(Z_j)_{j\in J}$ of $Z$.  As explained in \loccit, it coincides with Voevodsky's \cite[Definition 3.1.2]{Voehomo} in the noetherian case.

Let $S$ be a Stein space, and let $X$ be a separated $\cO(S)$-scheme 
of finite type. We will consider the following commutative diagram of sites:
\begin{equation}
\label{diagramofsites}
\begin{aligned}
\xymatrix@C=1.5em@R=3ex{
(X^{\an})_{\qc}\ar^{}[r]\ar^{\zeta}[d]&(X^{\an})'_{\cl}\ar^{\varepsilon}[d] \\
X_{\qfh}\ar[r]& (X_{\et})'.
}
\end{aligned}
\end{equation}

The site $(X_{\et})'$ is the variant of the site $X_{\et}$ of (\ref{carredesites}) where one restricts the objects to those \'etale $X$-schemes that are separated and of finite type over $\cO(S)$, and where one only considers coverings involving finitely many arrows. 
As objects of~$X_{\et}$ admit finite Zariski coverings by affine hence separated schemes, the topoi associated with~$(X_{\et})'$ and~$X_{\et}$ are equivalent by \cite[Lemma \href{https://stacks.math.columbia.edu/tag/03A0}{03A0}]{SP}.  We may thus use them interchangeably for cohomological computations. 

The site $(X^{\an})'_{\cl}$ is the variant of the site $(X^{\an})_{\cl}$ of (\ref{carredesites}) where one restricts the objects to those topological spaces endowed with a local homeomorphism to~$X^{\an}$ that are Hausdorff (but we impose no further constraints on coverings). As objects of~$(X^{\an})_{\cl}$ admit open coverings by Hausdorff spaces, the topoi associated with~$(X^{\an})'_{\cl}$ and~$(X^{\an})_{\cl}$ are equivalent, again by \cite[Lemma \href{https://stacks.math.columbia.edu/tag/03A0}{03A0}]{SP}. 

The site $X_{\qfh}$ has as objects those $X$-schemes that are separated and of finite type over $\cO(S)$, and is endowed with the $\qfh$ topology \cite[\S 6.6]{Chough}, restricting to coverings involving only finitely many arrows. 

The site $(X^{\an})_{\qc}$ has as objects the Hausdorff and locally compact topological spaces over $X^{\an}$ and is endowed with the $\qc$ topology \cite[Definition~\href{https://stacks.math.columbia.edu/tag/09X0}{09X0}]{SP}. 

The morphisms of sites in (\ref{diagramofsites}) are the obvious ones (\'etale coverings are $\qfh$ coverings by \cite[Lemma \href{https://stacks.math.columbia.edu/tag/0ETV}{0ETV}]{SP} and analytifications of $\qfh$ coverings are $\qc$ coverings as a consequence of \cite[Lemma~\href{https://stacks.math.columbia.edu/tag/09X5}{09X5}]{SP}).

\subsection{Artin's comparison theorem over Stein compact sets}
\label{parArtincompact}

We may now state and prove the main theorem of this article.

\begin{thm}
\label{Artincompcx}
Let $S$ be a Stein space. Let ${f:X\to\Spec(\cO(S))}$ be an $\cO(S)$\nobreakdash-scheme interiorly of finite type.
Fix a torsion abelian sheaf $\L$ on $X_{\et}$. Assume that $f$ is interiorly proper or that $\L$ is interiorly constructible. 
If one lets $U$ run over all Stein open neighborhoods of a Stein compact subset $K$ of $S$, the base change morphisms
\begin{equation}
\label{compaU}
\underset{K\subset U}{\colim}\,\,H^k_{\et}(X_{\cO(U)},\L_{\cO(U)})\to\underset{K\subset U}{\colim}\,\,H^k((X_{\cO(U)})^{\an},\L^{\an})
\end{equation}
are isomorphisms for $k\geq 0$.
\end{thm}

\begin{proof}
\setcounter{Step}{0}
We split the proof in several steps.

\begin{Step}
\label{step1bis}
We first reduce to the case where $X=\Spec(\cO(S))$.
\end{Step}

By Lemma \ref{semianal} (ii), the morphism (\ref{compaU}) identifies with the natural morphism
\begin{equation}
\label{compaL}
\underset{K\subset L}{\colim}\,\,H^k_{\et}(X_{\cO(L)},\L_{\cO(L)})\to\underset{K\subset L}{\colim}\,\,H^k((X_{\cO(L)})^{\an},\L^{\an}),
\end{equation}
where $L$ runs over all excellent Stein compact neighborhoods of $K$ in $S$, and
 $(X_{\cO(L)})^{\an}:=(f^{\an})^{-1}(L)$.
Both sides of (\ref{compaU}) or equivalently (\ref{compaL}) are computed by (colimits of) Leray spectral sequences, whose $E_2^{p,q}$ terms read respectively
$$\underset{K\subset L}{\colim}\,\,H^p_{\et}(\Spec(\cO(L)),(\RR^qf_*\L)_{\cO(L)})\textrm{ \,and \,}\underset{K\subset U}{\colim}\,\,H^p(U,(\RR^qf_*\L)^{\an}),$$
where we used Lemma \ref{cdbK} and Theorem \ref{relArtin} applied with ${Y=\Spec(\cO(S))}$. These terms are isomorphic by the $X=\Spec(\cO(S))$ case  of Theorem \ref{Artincompcx} applied to the sheaves $\RR^qf_*\L$.

\begin{Step}
\label{step2bis}
We further reduce to the case where $\L=\Z/m$ for some $m\geq 1$.
\end{Step}

The sheaf $\L$ is a filtered colimit of constructible sheaves by \cite[Lemma \href{https://stacks.math.columbia.edu/tag/03SA}{03SA}~(2)]{SP}.  In addition, if $L$ is an excellent Stein compact neighborhood of $K$ in $S$, both the \'etale cohomology of $\Spec(\cO(L))$ and the sheaf cohomology of $L$ commute with filtered colimits (see \cite[Theorem \href{https://stacks.math.columbia.edu/tag/09YQ}{09YQ}]{SP} and \cite[III, Theorem 5.1]{Iversen}). Consequently,  in view of the description (\ref{compaL}) of the morphism (\ref{compaU}), we may assume that~$\L$ is constructible.

By \cite[Lemmas \href{https://stacks.math.columbia.edu/tag/09Z7}{09Z7}, \href{https://stacks.math.columbia.edu/tag/095R}{095R} and \href{https://stacks.math.columbia.edu/tag/03RX}{03RX}]{SP}, the sheaf $\L$ has a resolution
\begin{equation}
\label{resolution2}
0\to\L\to\L_0\to\L_1\to\dots
\end{equation}
by constructible sheaves $\L_p$ that are finite products of sheaves of the form $\pi_*\M$ with $\pi:Z\to X$ finite of finite presentation and~$\M$ constant constructible on $Z_{\et}$.  Making use of the spectral sequences
\begin{alignat}{4}
&E_1^{p,q}=\underset{K\subset U}{\colim}\,\,H^q_{\et}(X_{\cO(U)},(\L_p)_{\cO(U)})&\implies& \underset{K\subset U}{\colim}\,\,H^{p+q}_{\et}(X_{\cO(U)},\L_{\cO(U)})\textrm{\hspace{1em} and }\nonumber\\
&E_1^{p,q}=\underset{K\subset U}{\colim}\,\,H^q((X_{\cO(U)})^{\an},\L_p^{\an})&\implies& \underset{K\subset U}{\colim}\,\,H^{p+q}((X_{\cO(U)})^{\an},\L^{\an})\nonumber
\end{alignat}
associated with the resolution (\ref{resolution2}), we reduce to the case where $\L$ is of the form~$\pi_*\M$. 
As the higher direct images of the morphisms induced by $\pi$ and $\pi^{\an}$ vanish by \cite[Proposition~\href{https://stacks.math.columbia.edu/tag/03QP}{03QP}]{SP} and \cite[III, Theorem 6.2]{Iversen}, and since one has ${\pi_*(\M)^{\an}\isoto\pi^{\an}_*(\M^{\an})}$ by Theorem \ref{relArtin},  we may further assume that $\L$ is constant (after replacing $S$ with~$Z^{\an}$ and~$\L$ with~$\M$), and then that $\L\simeq \Z/m$.

\begin{Step}
\label{step3bis}
We finally assume that $X=\Spec(\cO(S))$ and $\L=\Z/m$ for some $m\geq 1$.
\end{Step}

For any Stein open neighborhood $U$ of $K$ in $S$ (\resp any excellent Stein compact neighborhood $L$ of $K$ in $S$), write $X_U:=\Spec(\cO(U))$ (\resp $X_L:=\Spec(\cO(L))$). 
Our goal is to show that the morphisms
\begin{equation}
\label{goal}
\underset{K\subset U}{\colim}\,\,H^k_{\et}(X_U,\Z/m)\xrightarrow{\varepsilon^*}\underset{K\subset U}{\colim}\,\,H^k(X_U^{\an},\Z/m)
\end{equation}
are isomorphisms.  For any excellent Stein compact neighborhood $L$ of $K$ in $S$, the change of topology morphisms $H^k_{\et}(X_L,\Z/m)\to H^k((X_L)_{\qfh},\Z/m)$ are isomorphisms by \cite[Theorem 3.4.4]{Voehomo}. Taking the colimit over all such $L$ and using Lemma \ref{semianal} (ii) shows that 
$\underset{K\subset U}{\colim}\,\,H^k_{\et}(X_U,\Z/m)\to\underset{K\subset U}{\colim}\,\,H^k((X_U)_{\qfh},\Z/m)$
is an isomorphism.  As the pull-back morphisms $H^k(X_U^{\an},\Z/m)\to H^k((X_U^{\an})_{\qc},\Z/m)$ are also isomorphisms by \cite[Lemma \href{https://stacks.math.columbia.edu/tag/09X4}{09X4}]{SP}, it follows from diagram (\ref{diagramofsites}) that~(\ref{goal}) may be identified with the morphism
\begin{equation}
\label{reformulation}
\underset{K\subset U}{\colim}\,\,H^k((X_U)_{\qfh},\Z/m)\xrightarrow{\zeta^*}\underset{K\subset U}{\colim}\,\,H^k((X_U^{\an})_{\qc},\Z/m).
\end{equation}
The colimit over $U$ of the Leray spectral sequences of $\zeta$ reads
\begin{equation}
\label{colimitss}
E_2^{p,q}=\underset{K\subset U}{\colim}\,\,H^p((X_U)_{\qfh},\RR^q\zeta_*\Z/m)\implies \underset{K\subset U}{\colim}\,\,H^{p+q}(U_{\qc},\Z/m).
\end{equation}
We will prove that the adjunction morphism $\underset{K\subset U}{\colim}\,\,H^p((X_U)_{\qfh},\Z/m)\to E_2^{p,0}$ is an isomorphism in Step \ref{step4bis} and that $E_2^{p,q}=0$ for $q>0$ in Step \ref{step5bis}.  It follows that the spectral sequence (\ref{colimitss}) degenerates and that (\ref{reformulation}) hence (\ref{goal}) are isomorphisms.

\begin{Step}
\label{step4bis}
One has $\underset{K\subset U}{\colim}\,\,H^p((X_U)_{\qfh},\Z/m)\isoto\underset{K\subset U}{\colim}\,\,H^p((X_U)_{\qfh},\zeta_*\Z/m)$.
\end{Step}

 Let $U_0$ be a fixed Stein open neighborhood of $K$ in $S$ and let $Y$ be a separated $\cO(U_0)$\nobreakdash-scheme of finite type.  For any Stein open neighborhood $U$ of $K$ in $U_0$, write $Y_U:=Y\times_{X_{U_0}}X_{U}$. Consider the adjunction morphism
\begin{equation}
\label{isoBing1}
\underset{K\subset U\subset U_0}{\colim}\,\,H^0((Y_U)_{\qfh},\Z/m)\to\underset{K\subset U\subset U_0}{\colim}\,\,H^0((Y_U)_{\qfh},\zeta_*\Z/m).
\end{equation}
As the \'etale sheaf $\Z/m$ is already a $\qfh$ sheaf (see \cite[Lemma \href{https://stacks.math.columbia.edu/tag/0EW8}{0EW8}]{SP}),
the left-hand side of (\ref{isoBing1}) is isomorphic to $\underset{K\subset U\subset U_0}{\colim}\,\,H^0_{\et}(Y_U,\Z/m)$. In addition, as the constant sheaf~$\Z/m$ for the usual topology is already a $\qc$ sheaf (see \cite[Lemma \href{https://stacks.math.columbia.edu/tag/09X3}{09X3}]{SP}),  the right-hand side of (\ref{isoBing1}) is isomorphic to $\underset{K\subset U\subset U_0}{\colim}\,\,H^0((Y_U)^{\an},\Z/m)$. All in all, the morphism (\ref{isoBing1}) may be identified with 
\begin{equation*}
\label{isoBing}
\underset{K\subset U\subset U_0}{\colim}\,\,H_{\et}^0(Y_U,\Z/m)\to\underset{K\subset U\subset U_0}{\colim}\,\,H^0((Y_U)^{\an},\Z/m),
\end{equation*}
and hence is an isomorphism by \cite[(7.2)]{Bingener}.

 The computation of the cohomology of the site $(X_U)_{\qfh}$ using hypercoverings \cite[Proposition \href{https://stacks.math.columbia.edu/tag/09VZ}{09VZ}]{SP}, the fact that any given $\qfh$ hypercovering of $X_U$ involves only finitely many separated $\cO(U)$-schemes of finite type at each simplicial level (as coverings in $(X_U)_{\qfh}$ involve only finitely many arrows) and the fact that~(\ref{isoBing1}) is an isomorphism for such schemes, together imply that the natural morphism 
$$\underset{K\subset U}{\colim}\,\,H^p((X_U)_{\qfh},\Z/m)\to\underset{K\subset U}{\colim}\,\,H^p((X_U)_{\qfh},\zeta_*\Z/m)$$
is an isomorphism, which is what we wanted to prove.

\begin{Step}
\label{step5bis}
One has $\underset{K\subset U}{\colim}\,\,H^p((X_U)_{\qfh},\RR^q\zeta_*\Z/m)=0$ for $q>0$.
\end{Step}

Fix $q>0$. We claim that for any Stein open neighborhood $U_0$ of $K$ in $S$,  and any separated quasi-finite $\cO(U_0)$-scheme of finite presentation $Y$,  the group 
\begin{equation}
\label{claim0}
\underset{K\subset U\subset U_0}{\colim}\,\,H^0((Y_U)_{\qfh},\RR^q\zeta_*\Z/m)
\end{equation}
vanishes. Taking this claim for granted, the computation of the cohomology of the site $(X_U)_{\qfh}$ using hypercoverings \cite[Proposition \href{https://stacks.math.columbia.edu/tag/09VZ}{09VZ}]{SP}, the fact that any given $\qfh$ hypercovering of $X_U$ involves only finitely many separated quasi-finite $\cO(U)$\nobreakdash-schemes of finite presentation at each simplicial level, and the fact that~(\ref{claim0}) vanishes for such schemes, together imply that $\underset{K\subset U}{\colim}\,\,H^p((X_U)_{\qfh},\RR^q\zeta_*\Z/m)=0$.

We now prove the claim.  Using Lemma \ref{neighneigh}, we may assume that $K$ admits a basis $(U_i)_{i\in I}$ of Stein open neighborhoods such that $K$ is $\cO(U_i)$-convex for all~${i\in I}$.  
We may thus restrict the colimit (\ref{claim0}) to such neighborhoods.
Fix a Stein open neighborhood $U$ of $K$ in $U_0$ 
such that $K$ is $\cO(U)$-convex 
and choose a class $\alpha$ in $H^0((Y_U)_{\qfh},\RR^q\zeta_*\Z/m)$. There exists a $\qfh$ covering of $Y_U$, which we may choose to be given by a single morphism $W\to Y_U$, such that $\alpha$ is induced by a class $\beta\in H^q(W^{\an},\Z/m)$. By Zariski's Main Theorem (see \cite[Lemma \href{https://stacks.math.columbia.edu/tag/05K0}{05K0}]{SP}), there exists a factorization $W\xrightarrow{j}\oW\xrightarrow{\pi} Y_U$ of $W\to Y_U$ where~$j$ is an open immersion and $\pi$ is finite. 

By Proposition \ref{thkillram} applied to the Stein space~$(\oW)^{\an}$, to the Stein compact subset $(\pi^{\an})^{-1}(K)$, and to the cohomology class~$\beta$, there exist, after maybe shrinking $U$, a finite surjective holomorphic map $p:T\to(\oW)^{\an}$ and a class $\gamma\in H^q(T,\Z/m)$ with $\gamma|_{p^{-1}(W^{\an})}=\beta|_{p^{-1}(W^{\an})}$.  As $K$ is $\cO(U)$-convex,  the compact subset $(\pi^{\an}\circ p)^{-1}(K)$ is $\cO(T)$-convex (see Proposition \ref{propfinitehull} (i)). One can thus apply Proposition~\ref{killunram} to ensure, after modifying $p:T\to (\oW)^{\an}$ and further shrinking~$U$,  that $\gamma=0$ and hence that $\beta|_{p^{-1}(W^{\an})}=0$. 

By Lemma \ref{finiteanal},  the scheme $\Spec(\cO(T))$ is interiorly finite with analytification isomorphic to $T$ (both when viewed as a $\oW$-scheme or as an $\cO(U)$\nobreakdash-scheme), and the structural morphism $g:\Spec(\cO(T))\to\oW$ satisfies $g^{\an}=p$.  Using Lemma~\ref{int} (iii), one may assume after shrinking $U$ that $g$ is finite of finite presentation. 
 It follows that $g|_{g^{-1}(W)}: g^{-1}(W)\to W$ is a qfh covering. As $\beta|_{g^{-1}(W)^{\an}}=0$, we conclude that $\alpha\in H^0((Y_U)_{\qfh},\RR^q\zeta_*\Z/m)$ is qfh-locally trivial and hence trivial.
\end{proof}

When $f$ is interiorly proper, the statement of Theorem \ref{Artincompcx} takes a simpler form.

\begin{cor}
Let $K$ be a Stein compact subset of a Stein space $S$ and let $X$ be an interiorly proper $\cO(S)$\nobreakdash-scheme.
Fix a torsion abelian sheaf $\L$ on $X_{\et}$. For ${k\geq 0}$, there are canonical isomorphisms
$H^k_{\et}(X_{\cO(K)},\L|_{X_{\cO(K)}})\to H^k((X_{\cO(K)})^{\an},\L^{\an})$.
\end{cor}

\begin{proof}
The isomorphism (\ref{compaU}) has the required form by \cite[Theorem \href{https://stacks.math.columbia.edu/tag/09YQ}{09YQ}]{SP} and \cite[III, Lemma 6.3]{Iversen}. 
\end{proof}

\begin{rem}
Bingener's \cite[Theorem 7.4]{Bingener} implies the particular case of Theorem \ref{Artincompcx} where $k=1$,  the sheaf $\L$ is constant and $X=\Spec(\cO(S))$.
\end{rem}

\subsection{\texorpdfstring{$G$}{G}-equivariant complex spaces}
\label{Geqpar}

 Let $G:=\Gal(\C/\R)\simeq\Z/2$ be the Galois group of $\R$,  generated by the complex conjugation $\sigma\in G$.
We use the conventions of \cite[Appendix A]{tight} concerning $G$-equivariant complex-analytic geometry.
In particular, a $G$-equivariant complex space is a complex space endowed with an action of $G$ (as a locally ringed space) such that $\sigma\in G$ acts $\C$-antilinearly.  

A $G$-equivariant complex space is said to be Stein if so is its underlying complex space.  A $G$-invariant Stein compact subset of a $G$-equivariant complex space $S$ admits a basis of $G$-invariant Stein open neighborhoods in $S$ (as the intersection of a Stein open neighborhood of $K$ in $S$ with its image by $\sigma$ is Stein).

The \textit{complex conjugate} $S^{\sigma}$ of a complex space $S$ with structural morphism ${\mu:\C\to\cO_S}$ is the complex space which is equal to $S$ as a locally ringed space, with structural morphism $\mu_{S^{\sigma}}:=\mu\circ\sigma:\C\to\cO_S$. A $G$-equivariant complex space can equivalently be described as a complex space $S$ endowed with an isomorphism $\alpha:S^{\sigma}\isoto S$ such that $\alpha\circ\alpha^{\sigma}=\Id_S$ (see \cite[\S A.2]{tight}).

Note that $\cO(S)=\cO(S)^G\otimes_{\R}\C$. A morphism $f:X\to Y$ of $\cO(S)^G$\nobreakdash-schemes is said to interiorly have a property if such is the case for $f_{\cO(S)}:X_{\cO(S)}\to Y_{\cO(S)}$.
If $X$ is an $\cO(S)^G$-scheme interiorly locally of finite type, the descent datum on $X_{\cO(S)}=X\times_{\Spec(\R)}\Spec(\C)$ yields an isomorphism $\alpha:((X_{\cO(S)})^{\an})^{\sigma}\isoto (X_{\cO(S)})^{\an}$ which endows $(X_{\cO(S)})^{\an}$ with a structure of $G$-equivariant complex space: the analytification $X^{\an}$ of $X$. The site morphism $\varepsilon:(X^{\an})_{\cl}\to X_{\et}$ of (\ref{carredesites}) is $G$\nobreakdash-equivariant for the actions of $G$ by complex conjugation. In particular, the analytification $\L^{\an}$ of a sheaf $\L$ on $X_{\et}$ is naturally a $G$-equivariant sheaf on $X^{\an}$. We will say that~$\L$ is interiorly constructible if so is $\L|_{X_{\cO(S)}}$.

We state for later use in the proof of Theorem \ref{sosArtincx} a $G$-equivariant analogue of Corollary \ref{eqcatetale}, which appears in \cite[Proposition A.5]{tight} when $S$ has dimension~$1$.

\begin{prop}
\label{eqcatG}
Let $S$ be a reduced $G$-equivariant Stein space with finitely many irreducible components. 
Associating with $p:T\to S$ the $\cM(S)^G$\nobreakdash-algebra $\cM(T)^G$ induces an equivalence of categories
$$
 \left\{  \begin{array}{l}
    \textrm{$G$-equivariant analytic coverings }p:T\to S\\
    \textrm{\hspace{6em}with $T$ normal}
  \end{array}\right\}\to
 \left\{  \begin{array}{l}
    \textrm{\hspace{1.3em}finite \'etale}\\\cM(S)^G\textrm{-algebras}
  \end{array}\right\}.
$$
\end{prop}

\begin{proof}
The proposition follows from Corollary \ref{eqcatetale} by using the above description of $G$\nobreakdash-equivariant complex analytic spaces as complex analytic spaces $S$ endowed with an isomorphism $\alpha:S^{\sigma}\isoto S$ such that $\alpha\circ\alpha^{\sigma}=\Id_S$.
\end{proof}

\subsection{Artin's comparison theorem over \texorpdfstring{$G$}{G}-equivariant Stein compact sets}
\label{Geqpar2}

  Here is an extension of Theorem \ref{Artincompcx} to the $G$-equivariant setting, which goes back to Cox \cite[Theorem 1.1]{Cox} when $S$ is a point (see also \cite[(15.3.2)]{Scheiderer}).

\begin{thm}
\label{Artincompreal}
Let $S$ be a $G$-equivariant Stein space. Let ${f:X\to\Spec(\cO(S)^G)}$ be an $\cO(S)^G$\nobreakdash-scheme interiorly of finite type.
Fix a torsion abelian sheaf $\L$ on $X_{\et}$. Assume that $f$ is interiorly proper or that $\L$ is interiorly constructible.
If one lets~$U$ run over all $G$-invariant Stein open neighborhoods of a $G$-invariant Stein compact subset $K$ of $S$, the base change morphisms
\begin{equation}
\label{compaU2}
\underset{K\subset U}{\colim}\,\,H^k_{\et}(X_{\cO(U)^G},\L_{\cO(U)^G})\to\underset{K\subset U}{\colim}\,\,H_G^k((X_{\cO(U)^G})^{\an},\L^{\an})
\end{equation}
are isomorphisms for $k\geq 0$.
\end{thm}

\begin{proof}
The $G$-equivariant site morphisms $\varepsilon:((X_{\cO(U)})^{\an})_{\cl}\to (X_{\cO(U)})_{\et}$ induces morphisms between the Hochschild--Serre spectral sequences of \cite[Remark~10.9]{Scheiderer}
\begin{alignat}{4}
&E_2^{p,q}=H^p(G,H^q_{\et}(X_{\cO(U)},\L_{\cO(U)}))&\implies& H^{p+q}_{\et}(X_{\cO(U)^G},\L_{\cO(U)^G})\textrm{\hspace{.3em} and }\nonumber\\
&E_2^{p,q}=H^p(G,H^q((X_{\cO(U)})^{\an},\L^{\an}))&\implies& H^{p+q}_G((X_{\cO(U)^G})^{\an},\L^{\an}).\nonumber
\end{alignat}
After taking the colimit over $U$,  we get an isomorphism on page $2$ by Theorem \ref{Artincompcx},  hence an isomorphism on the abutment.
\end{proof}

\subsection{Cohomological dimension} 
\label{cdpar}

A scheme~$X$ is said to have \textit{cohomological dimension} $\leq n$ if the \'etale cohomology of any torsion abelian sheaf on~$X_{\et}$ vanishes in degree $>n$.
Our next theorem controls the cohomological dimension of affine varieties over Stein compacta. It is a generalization in Stein geometry of the bounds on the cohomogical dimension of complex and real affine varieties obtained in \cite[XIV, Corollaire~3.2]{SGA43} and \cite[Corollary 7.21]{Scheiderer}.

\begin{thm}
\label{cdaffineG}
Let $K$ be a $G$-invariant Stein compact subset in a $G$-equivariant Stein space $S$. Let $X$ be an affine $\cO(S)^G$-scheme interiorly of finite type with $(X^{\an})^G=\varnothing$. Then $X_{\cO(K)^G}$ has \'etale cohomological dimension~$\leq \dim(X^{\an})$.
\end{thm}

\begin{proof}
We may suppose that $\dim(X^{\an})<\infty$. We wish to show the vanishing of the degree $k$ cohomology of a torsion abelian sheaf on $(X_{\cO(K)^G})_{\et}$  if $k>\dim(X^{\an})$.
By~\cite[Lemma \href{https://stacks.math.columbia.edu/tag/03SA}{03SA} (2)]{SP}, any such sheaf is a filtered colimit of constructible sheaves.  In view of \cite[Lemma \href{https://stacks.math.columbia.edu/tag/03Q5}{03Q5}~(2)]{SP}, we may thus only consider constructible sheaves.
By \cite[Lemma \href{https://stacks.math.columbia.edu/tag/09YU}{09YU}]{SP},  any constructible abelian sheaf on~$(X_{\cO(K)^G})_{\et}$ is of the form~$\L|_{X_{\cO(K)^G}}$ for some constructible abelian sheaf~$\L$ on $X_{\et}$,  after possibly shrinking $S$.
Letting~$U$ run over all $G$-invariant Stein open neighborhoods of $K$ in~$S$, one computes
\begin{equation}
\begin{alignedat}{5}
\label{eqcd}
H^k_{\et}(X_{\cO(K)^G},\L|_{X_{\cO(K)^G}})&=\underset{K\subset U}{\colim}\,\,H^k_{\et}(X_{\cO(U)^G},\L|_{X_{\cO(U)^G}})\\
&=\underset{K\subset U}{\colim}\,\,H^k_G((X_{\cO(U)^G})^{\an},\L^{\an}|_{(X_{\cO(U)^G})^{\an}}),
\end{alignedat}
\end{equation}
where we used successively \cite[Theorem \href{https://stacks.math.columbia.edu/tag/09YQ}{09YQ}]{SP}, and Theorem \ref{Artincompreal}.

After shrinking $S$, we may assume that the affine $\cO(S)^G$-scheme $X$ is of finite presentation,  and hence is defined in $\A^N_{\cO(S)^G}$ for some $N\geq 0$ by the vanishing of finitely many elements of $\cO(S)^G[x_1,\dots,x_N]$. It follows from the concrete construction of the analytification recalled in \S\ref{parBingener} that the $G$\nobreakdash-equivariant complex space~$(X_{\cO(U)^G})^{\an}$ may be realized as a $G$-invariant closed complex subspace of~$\C^N\times U$.  It is therefore a $G$\nobreakdash-equivariant Stein space.  

As $\L$ is constructible, the sheaf $\L^{\an}$ on $X^{\an}$ is weakly constructible (in the sense that it is locally constant in restriction to the strata of some analytic stratification of $X^{\an}$).  By our hypothesis that $(X^{\an})^G=\varnothing$, the Artin vanishing theorem of \cite[Theorem 2.6]{Artinvanishing} implies that ${H^k_G((X_{\cO(U)^G})^{\an},\L^{\an}|_{(X_{\cO(U)^G})^{\an}})=0}$ for~${k>\dim(X^{\an})}$.
It now follows from (\ref{eqcd}) that $H^k_{\et}(X_{\cO(K)^G},\L|_{X_{\cO(K)^G}})=0$ for $k>\dim(X^{\an})$.
\end{proof}

The following application of Theorem \ref{cdaffineG} extends an algebraic result  \cite[Proposition 1.2.1]{CTP} attributed to Ax by Colliot-Th\'el\`ene and Parimala.

\begin{thm}
\label{cdfieldG}
Let $K$ be a $G$-invariant Stein compact subset of a $G$-equivariant Stein space $S$.  Let $X$ be an $\cO(S)^G$-scheme interiorly of finite type with ${(X^{\an})^G=\varnothing}$.  If $X_{\cO(K)^G}$ is integral, its function field has cohomological dimension~$\leq \dim(X^{\an})$.
\end{thm}

\begin{proof}
We may assume that $X$ is affine (after replacing it with an affine open subset~$X'\subset X$ such that $X'_{\cO(K)^G}$ is nonempty).
Let $(U_i)_{i\in\N}$ be a decreasing basis of $G$-invariant Stein open neighborhoods of~$K$ in $S$. 
If $i\in \N$ and $h\in \cO(X_{\cO(U_i)^G})$ is nonzero in restriction to $X_{\cO(K)^G}$, we define $Y_{i,h}$ to be the distinguished affine open subset of $X_{\cO(U_i)^G}$ where $h\neq 0$.
Then the function field $F$ of~$X_{\cO(K)^G}$ can be written as the filtered colimit of rings $F=\underset{i,h}{\colim}\,\,\cO((Y_{i,h})_{\cO(K)^G})$. By Theorem~\ref{cdaffineG} applied to the affine $\cO(U_i)^G$-scheme $Y_{i,h}$, the scheme $(Y_{i,h})_{\cO(K)^G}$ has cohomological dimension $\leq\dim(X^{\an})$. We deduce from \cite[Lemma \href{https://stacks.math.columbia.edu/tag/0F0R}{0F0R}]{SP} that $F$ has cohomological dimension $\leq\dim(X^{\an})$. 
\end{proof}

If $(X^{\an})^G\neq\varnothing$, one may sometimes apply Theorem \ref{cdfieldG} to a Zariski open subset~$X'$ of $X$ such that $((X')^{\an})^G=\varnothing$. The next corollary is obtained in this way.

\begin{cor}
\label{cdfieldGcor}
Let $K$ be a $G$-invariant Stein compact subset of a $G$-equivariant Stein space $S$.
If $S^G$ is contained in a nowhere dense closed analytic subset ${Z\subset S}$ and $\cO(K)^G$ is a domain, then $\Frac(\cO(K)^G)$ has cohomological dimension~$\leq\dim(S)$.
\end{cor}

\begin{proof}
Define $V:=\Spec(\cO(S)^G)\setminus \Spec(\cO(S)^G/\cI_Z(S)^G)$.  By Lemma~\ref{openanal} and its proof, the $\cO(S)^G$-scheme $V$ is interiorly of finite type and the analytification of the open immersion $V\hookrightarrow \Spec(\cO(S)^G)$ identifies with $S\setminus Z\hookrightarrow S$.  It now suffices to apply Theorem \ref{cdfieldG} with $X=V$ to conclude.
\end{proof}

If $S$ is a complex space, then the disjoint union $T:=S\sqcup S^{\sigma}$ may be endowed with a structure of $G$-equivariant complex space, by letting $\sigma\in G$ exchange the two factors (and act by the identity on the underlying locally ringed spaces).  One then has $T^G=\varnothing$ and $\cO(T)^G\simeq\cO(S)$. Using this trick, one can formally deduce non-$G$-equivariant statements from $G$\nobreakdash-equivariant ones.  For instance,  Theorems \ref{cdaffineG} and \ref{cdfieldG} and Corollary \ref{cdfieldGcor} immediately imply the following.

\begin{thm}
\label{cdaffine}
Let $K$ be a Stein compact subset of a Stein space $S$.  Let $X$ be an affine $\cO(S)$-scheme interiorly of finite type. Then $X_{\cO(K)}$ has \'etale cohomological dimension~$\leq \dim(X^{\an})$.
\end{thm}

\begin{thm}
\label{cdfield}
Let $K$ be a Stein compact subset of a Stein space $S$.  Let $X$ be an $\cO(S)$-scheme interiorly of finite type. If $X_{\cO(K)}$ is integral, then its function field has cohomological dimension~$\leq \dim(X^{\an})$.
\end{thm}

\begin{cor}
\label{cdfieldcor}
Let $K$ be a Stein compact subset of a Stein space $S$.  If~$\cO(K)$ is a domain, then $\Frac(\cO(K))$ has cohomological dimension~$\leq \dim(S)$.
\end{cor}

\begin{rems}
\label{remcd}
(i)
One could have given direct proofs of Theorems \ref{cdaffine} and~\ref{cdfield} and Corollary \ref{cdfieldcor}, similar to the proofs of Theorems \ref{cdaffineG} and \ref{cdfieldG} and Corollary \ref{cdfieldGcor} (replacing Theorem \ref{Artincompreal} by Theorem~\ref{Artincompcx}).

(ii) 
In the setting of Corollary \ref{cdfieldcor}, if $S$ has pure dimension $n$ and $K$ is nonempty, the cohomological dimension of $\Frac(\cO(K))$ is equal to $n$. To see it, choose ${s\in K}$.
The local ring morphism $\cO(K)_s\to\cO_{S,s}$ is flat by Lemma \ref{flat}.   As the maximal ideal of $\cO(K)_s$ generates the maximal ideal of $\cO_{S,s}$ because $S$ is Stein,  it follows from \cite[Lemmas \href{https://stacks.math.columbia.edu/tag/033E}{033E} and \href{https://stacks.math.columbia.edu/tag/00ON}{00ON}]{SP} that $\cO(K)_s$ is noetherian of dimension $n$.
By \cite[X, Corollaire 2.5]{SGA43}, the field $\Frac(\cO(K))=\Frac(\cO(K)_s)$ has cohomological dimension~$\geq n$. 
\end{rems}

\section{Applications to Hilbert's $17$th problem}
\label{parsos}

We may now present our applications to sums of squares of analytic functions.

\subsection{Sums of squares}

The following theorem is an analogue of Artin's solution to Hilbert's $17$th problem \cite{Artin} on $G$-equivariant Stein compacta.

\begin{thm}
\label{sosArtincx}
Let $K$ be a $G$-invariant Stein compact subset of a reduced
$G$\nobreakdash-equi\-variant Stein space $S$. Let $f\in\cO(K)^G$ be nonnegative on a neighborhood of $K^G$ in~$S^G$.  Then $f$ is a sum of squares in~$\cM(K)^G$.
\end{thm}

\begin{proof}
Using \cite[8, \S 1.3, Proposition]{GRCoherent}, we may replace $S$ by its normalization and hence assume that it is normal.
Let $K'$ be a Stein compact neighborhood of~$K$ in $M$ such that $f\in\cO(K')^G$ (see Lemma \ref{semianal} (ii)). By compactness of $K$, we may assume that $K'$ has finitely many connected components. Replacing $K$ with a $G$-orbit of connected components of $K'$, we may assume that $K/G$ is connected.

We may suppose that $f\neq 0$. After shrinking $S$, we may assume that $S/G$ is connected, that $f\in\cO(S)^G$, and that $f$ is nonnegative on $S^G$. 
Let $p:T\to S$ be the $G$\nobreakdash-equivariant analytic covering of normal $G$-equivariant Stein spaces associated with the finite \'etale $\cM(S)^G$-algebra ${F:=\cM(S)^G[x]/\langle x^2+f\rangle}$ by Proposition \ref{eqcatG}. The nonnegativity hypothesis on $f$ and the fact that $-f$ is a square in $\cM(T)^G$,  hence in~$\cO(T)^G$ by normality of $T$, imply that $T^G$ is contained in the nowhere dense closed analytic subset $\{f=0\}$ of~$T$.  
Let $L\subset T$ be a $G$\nobreakdash-orbit of connected components of $p^{-1}(K)$. It is a $G$-invariant Stein compact subset of $T$ by \cite[V, \S 1.1, Theorem~1~d)]{GRStein}.  By Corollary~\ref{cdfieldGcor}, the field~$\cM(L)^G$ has finite cohomological dimension, and hence cannot be ordered (see e.g. \cite[Remark 7.5]{Scheiderer}).

 As $\cM(K)^G\subset\cM(L)^G\subset\cM(K)^G[x]/\langle x^2+f\rangle$, one has $\cM(L)^G=\cM(K)^G$ or $\cM(L)^G=\cM(K)^G[x]/\langle x^2+f\rangle$.  In the first case,  $f$ is a sum of squares in $\cM(K)^G$ by \cite[VIII, Theorem 1.10 and Proposition 1.1 (2)]{Lam}. In the second case,  $f$ is a sum of squares in $\cM(K)^G$ by \cite[VIII, Theorem 1.10 and Basic Lemma~1.4]{Lam}. 
\end{proof}

\begin{rem}
Theorem \ref{sosArtincx} can be proven by the more elementary methods of~\cite{Jaworski2}, but one could not recover in this way the quantitative results of Theorem~\ref{soscx}.
\end{rem}

\subsection{Sums of few squares}

If $A$ is a ring, we let $\cd_2(A)$ denote the \'etale cohomological $2$-dimension of $\Spec(A)$,  i.e.\ the largest integer $n$ such that there exists a $2$-primary torsion \'etale sheaf $\L$ on $\Spec(A)$ with $H^n_{\et}(\Spec(A),\L)\neq 0$ (or $+\infty$ if no such integer exists).

\begin{prop}
\label{propVoe}
Let $F$ be a field such that $\cd_2(F[x]/\langle x^2+1\rangle)\leq n$. If $f\in F$ is a sum of squares in $F$, then it is a sum of $2^n$ squares in~$F$.
\end{prop}

\begin{proof}
If $\car(F)=2$, then sums of squares in $F$ are squares in $F$, so we may assume that $\car(F)\neq 2$.
We may also assume that $f\neq 0$. Denote by $\{a\}\in H^1(F,\Z/2)$ the class induced by $a\in F^*$ via the Kummer isomorphism $F^*/(F^*)^2\isoto H^1(F,\Z/2)$.

As $f$ is a sum of squares, it is positive with respect to all the field orderings of~$F$. By a theorem of Arason \cite[Satz 3]{Arason},
the class $\{f\}\cdot\{-1\}^N\in H^{N+1}(F,\Z/2)$ vanishes for $N\gg0$. Consider the \'etale $F$-algebra $F':=F[x]/\langle x^2+1\rangle$ and set $\Gamma:=\Gal(F'/F)$. For all $N\geq 0$, the short exact sequence
$$0\to \Z/2\to\Z/2[\Gamma]\to\Z/2\to 0$$
of $\Gamma$-modules induces an exact sequence
$$H^N(F',\Z/2)\to H^N(F,\Z/2)\xrightarrow{\cdot\{-1\}}H^{N+1}(F,\Z/2).$$
A decreasing induction on the degree shows that $\{f\}\cdot\{-1\}^N=0$ for all $N\geq n$. The Milnor conjectures proven by Voevodsky \cite{Voevodsky} now imply that $f$ is a sum of~$2^n$ squares in $F$ (see \cite[Proposition~2.1]{Henselian}).
\end{proof}

The next theorem is a quantitative improvement of Theorem \ref{sosArtincx} in the spirit of Pfister's theorem \cite[Theorem 1]{Pfister}.

\begin{thm}
\label{soscx}
Let $K$ be a $G$-invariant Stein compact subset of a reduced $G$\nobreakdash-equi\-variant Stein space $S$ of dimension $n$.  Let $f\in\cO(K)^G$ be nonnegative on a neighborhood of $K^G$ in $S^G$.  Then $f$ is a sum of $2^n$ squares~in~$\cM(K)^G$.
\end{thm}

\begin{proof}
Arguing as in the proof of Theorem \ref{sosArtincx}, we may assume that $S$ is normal and that $K/G$ is connected. Then $F:=\cM(K)^G$ is a field and ${F[x]/\langle x^2+1\rangle\simeq \cM(K)}$ has cohomological dimension $\leq n$ by Corollary~\ref{cdfieldcor} applied to the connected components of $K$. Apply Theorem \ref{sosArtincx} and Proposition~\ref{propVoe} to conclude. 
\end{proof}

\subsection{Real-analytic geometry}

Theorem \ref{soscx} has applications to real-analytic variants of Hilbert's $17$th problem. 
We follow the conventions of \cite{GMT} and refer to \cite[II, Definition 1.4]{GMT} for the definitions of \textit{real-analytic spaces} and \textit{real-analytic varieties}. 
If $K$ is a closed subset of a real-analytic space $M$, we denote by $\cO(K)$ (\resp $\cM(K)$) the ring of germs of real-analytic functions (\resp of real-analytic meromorphic functions) in a neighborhood of $K$.

\begin{thm}
\label{sosreal}
Let $M$ be a real-analytic space of dimension $n$ with reduced local rings.
Let $K\subset M$ be a compact subset. Let $f\in \cO(K)$ be nonnegative on a neighborhood of~$K$ in $M$. Then $f$ is a sum of $2^n$ squares in $\cM(K)$.
\end{thm}

\begin{proof}
By \cite[III, Theorems 3.6 and~3.10]{GMT}, there exist a reduced $G$-equivariant Stein space $S$ of dimension $n$, and an isomorphism $M\isoto S^G$ of real-analytic spaces. In addition, there exists a proper injective $G$-equivariant holomorphic map $i:S\to \C^{N}$ for $N\gg0$ (see \cite[V, Theorem 3.7]{GMT}).

The compact subset $i(K)\subset \R^N$ admits a basis of Stein neighborhoods in $\C^N$ (see \cite[Lemma 5]{SiuNoeth}) which we may choose to be $G$-invariant as the intersection of two Stein open subsets is Stein (see \cite[p.~127]{GRStein}). We deduce that $K$ admits a basis of $G$\nobreakdash-invariant Stein neighborhoods in $S$ (use \cite[V, \S 1.1, Theorem~1~d)]{GRStein}). By \cite[III, Proposition 1.8]{GMT}, one may choose one such neighborhood~$S$ with the property that~$f$ extends to a holomorphic map $g:S\to \C$. After replacing $g$ with   $(g+\overline{g\circ \sigma})/2$, we may assume that it is $G$-equivariant. 
One may now apply Theorem \ref{soscx} to the function $g$ to conclude.
\end{proof}

\begin{rem}
Theorem \ref{sosreal} applies to normal real-analytic varieties of pure dimension by \cite[IV, Proposition 3.8]{GMT}, hence to real-analytic manifolds.
\end{rem}

\bibliographystyle{myamsalpha}
\bibliography{compactStein}

\end{document}